\documentclass[a4paper,10pt]{article}
\usepackage{cmap}
\usepackage[T1]{fontenc}
\usepackage[utf8]{inputenc}
\usepackage{amssymb}
\usepackage{amsmath}
\usepackage{amsthm}

\usepackage{cleveref}
\usepackage{graphicx}
\usepackage[disable]{todonotes} 
\usepackage{tikz-cd}
\usepackage{comment}

\newcommand{\Z}{\mathbb{Z}}
\newcommand{\R}{\mathbb{R}}
\newcommand{\C}{\mathbb{C}}
\newcommand{\HH}{\mathbb{H}}
\newcommand{\T}{\mathbb{T}}

\newcommand{\re}{\ensuremath{\mathop{\rm Re}\nolimits}} 
\newcommand{\im}{\ensuremath{\mathop{\rm Im}\nolimits}}

\newcommand{\dd}{\mathrm{d}}

\newcommand{\CP}{\mathbb{CP}}

\newcommand{\sff}{\mathrm{I\!I}}
\newcommand{\nk}{\overline{\nabla}} 

\newcommand{\SO}{\mathrm{SO}}
\renewcommand{\O}{\mathrm{O}}
\newcommand{\SU}{\mathrm{SU}}
\newcommand{\U}{\mathrm{U}}
\newcommand{\Sp}{\mathrm{Sp}}

\newcommand{\mf}[1]{\mathfrak{#1}}

\newcommand{\reP}{\ensuremath{\re \psi}}
\newcommand{\imP}{\ensuremath{\im \psi}}

\theoremstyle{plain} 
\newtheorem{thm}{Theorem}
\numberwithin{thm}{section}
\newtheorem*{thm*}{Theorem}
\newtheorem{lma}[thm]{Lemma} 
\newtheorem{crl}[thm]{Corollary}
\newtheorem{prop}[thm]{Proposition}
\newtheorem*{prop*}{Proposition}
\newtheorem{example}[thm]{Example}
\newtheorem*{conjecture*}{Conjecture}

\theoremstyle{remark}
\newtheorem{rmk}[thm]{Remark}
\newtheorem*{rmk*}{Remark}

\theoremstyle{definition}
\newtheorem{dfn}[thm]{Definition}
\newtheorem*{dfn*}{Definition}

\numberwithin{equation}{section}

\usepackage[top=1in, bottom=1.25in, left=1.25in, right=1.25in]{geometry}
\usepackage[backend=biber ,maxbibnames=99,style=alphabetic]{biblatex}
\newcommand{\cp}{\CP^3}
\title{Special Lagrangians in nearly K{\"a}hler $\mathbb{CP}^3$}
\author{Benjamin Aslan} 
\newcommand{\Addresses}{%
 \bigskip
 \footnotesize
 \textsc{Department of Mathematics, University of West London,
  UK}\par\nopagebreak
 \textit{E-mail address: }\texttt{benjamin.aslan.17@ucl.ac.uk}}
 \DeclareUnicodeCharacter{0301}{} 
\begin{document}
\maketitle
\begin{abstract}
 This article explores special Lagrangian submanifolds in $\mathbb{CP}^3$, viewed as a nearly Kähler manifold, from two different perspectives. Intrinsically, using a moving frame set-up, and extrinsically, using $\mathrm{SU}(2)$ moment-type maps. We describe new homogeneous examples, from both perspectives, and classify totally geodesic special Lagrangian submanifolds. We show that every special Lagrangian in $\mathbb{CP}^3$, or the flag manifold $\mathbb{F}_{1,2}(\mathbb{C}^3)$ admitting a symmetry of an $\SU(2)$ subgroup of nearly Kähler automorphisms is automatically homogeneous. 
\end{abstract}
Nearly Kähler manifolds were first introduced in the 1970s and in the last two decades fundamental questions about the structure and existence of these manifolds were settled \cite{nagy2002nearly,butruille2010homogeneous,foscolo2015new}, making them a trending topic in differential geometry. In dimension 6, they form a special class of $\SU(3)$-structures and are important for Riemannian geometry as they provide examples of Einstein manifolds. In addition, they are of interest in exceptional holonomy as their cones are torsion free $G_2$ manifolds, which makes nearly Kähler manifolds crucial for understanding $G_2$ manifolds with singularities. \par
One peculiarity of Lagrangian submanifolds of nearly Kähler manifolds is that they are automatically special Lagrangian.
Because of their simple definition they are natural objects to study in nearly Kähler geometry, following the strategy to understand an ambient space by studying its distinguished submanifolds.
Just as for $J$-holomorphic curves in nearly Kähler manifolds there are two additional lines of motivation to study Lagrangian submanifolds. The first one comes from Riemannian geometry, for any special Lagrangian in a nearly Kähler manifold is minimal. The second comes from special holonomy, for the cone of a special Lagrangian is coassociative in the $G_2$-cone of $M$.\par
In the last few decades, many constructions for special Lagrangian submanifolds of $S^6$ have been found and various subclasses of special Lagrangians have been classified, see for example \cite{vrancken2003special,lotay2011ruled}. More recently, the ambient spaces $\mathbb{F}$ and $S^3\times S^3$ have received attention, for example in \cite{bektacs2019lagrangian,storm2020lagrangian}. This article is dedicated to the ambient nearly Kähler space $M=\CP^3$. The main results of this article are the classification of totally geodesic special Lagrangians in \cref{prop: totally geodesic class} and of special Lagrangians admitting a symmetry of a 3-dimensional group of nearly Kähler automorphisms in \cref{thm: classification SL orbits}.
These results are obtained through two different approaches to describe special Lagrangians in $\cp$.\par
The first approach is intrinsic as it uses the structure equations describing a special Lagrangian. We derive them in the general nearly Kähler setting in \cref{sec: structure eqns in nk manifolds} and show in \cref{prop: bonnet} that for a homogeneous ambient space and a simply connected domain there is a unique Lagrangian immersion for every solution to the structure equations. \par
In \cref{sec: angle function SL}, we adapt these equations to the twistor fibration $\cp \to S^4$.
We introduce an angle function $\theta\colon L \to [0,\frac{\pi}{4}]$ parametrising the Lagrangian at a tangent level. Generically, $L$ intersects every twistor fibre transversally. The points where $\theta=\frac{\pi}{4}$ are those where this is not the case, and the intersection is then diffeomorphic to a circle. We identify Lagrangians with $\theta\equiv \frac{\pi}{4}$ as circle bundles over superminimal surfaces in $S^4$, a construction discovered in \cite{storm2020note} and in \cite{konstantinov2017higher}.\par
Finally, we classify all Lagrangians where $\theta$ takes the boundary value $0$. In fact, there are just two such examples and they are both homogeneous. We describe another somewhat surprising homogeneous example with $\theta\equiv\frac{1}{2}\arccos(\frac{7 \sqrt{2}}{5 \sqrt{5}})$ arising from the irreducible representation of $\SU(2)$ on $S^3(\C^2)$. We also show that the standard $\mathbb{RP}^3$ in $\mathbb{CP}^3$ is the only totally geodesic Lagrangian submanifold of $\CP^3$. \par
The second approach to exploring special Lagrangians in $\mathbb{CP}^3$ is extrinsic. In \cref{sec: SU2 SL}, we introduce $\SU(2)$ moment maps in nearly Kähler geometry. They encode the symmetry of the nearly Kähler manifold in a set of $\SU(2)$ equivariant functions $M\to \R^3\oplus \R$. We use these moment-type maps to show a general existence result of special Lagrangians with $\SU(2)$ symmetry, in \cref{crl: existence SL orbits} and to classify special Lagrangians admitting an action of a $\SU(2)$ group of automorphisms, in \cref{thm: classification SL orbits}. We show that they are, in fact, all homogeneous and describe the examples found in \cref{sec: angle function SL} extrinsically. We show an analogous result for the flag manifold $\mathbb{F}$ by studying the action of three-dimensional subgroups of $\SU(3)$ on $\mathbb{F}$.\par
Most of the material presented in this article originates from author's PhD thesis, \cite{aslan2022special}.
\\
\vspace{0.1cm}
\\
\begin{minipage}{\linewidth}
\textbf{Acknowledgement.} The author is grateful to Jason Lotay, Simon Salamon and Thomas Madsen for their advice and support. The review from Lorenzo Foscolo and Luc Vrancken of the author's PhD thesis has led to a significant improvement of the material.
This work was supported by the London Mathematical Society grant ECF-2021-01.
\end{minipage}
\\
\vspace{0.1cm}
\\
\section{Background}
 Let $(M,g,J,\omega)$ be a 6-dimensional almost Hermitian manifold. Then $M$ is called nearly Kähler provided there is a complex-valued three-form $\psi=\reP+i \imP\in \Lambda^{3,0}(M)$ defining an $\mathrm{SU}(3)$-structure satisfying
 \begin{align*}
  \dd \omega&=3  \reP \\
  \dd \imP&=-2  \omega\wedge \omega.
 \end{align*}
\par
\begin{rmk}
 Often, the Calabi-Yau case $\dd \omega=0, \dd \imP=0$ is also included in the definition of a nearly Kähler manifold. We choose to exclude this case, so there is no need to introduce the subclass of strictly nearly Kähler manifolds.
\end{rmk}
Every nearly Kähler manifold admits a unique connection $\nk$ with totally skew-symmetric torsion and holonomy contained in $\mathrm{SU}(3)$, i.e. $\nk g =\nk J =\nk \psi=0$, cf. \cite{gray1970nearly}. This connection is called the characteristic connection and related to the Levi-Civita connection by 
\begin{align}
\label{eqn: nk vs lc}
  g(\nk_X Y,Z)=g(\nabla_X Y,Z)+\frac{1}{2}\reP(X,Y,JZ),
\end{align}
see \cite{nkdeformations}. Examples of (compact) nearly Kähler manifolds are very scarce. In fact, there are only six known examples of compact simply-connected nearly Kähler manifolds.
\begin{prop}\upshape{\cite[Theorem 1]{butruille2010homogeneous}}
\label{prop: nk class}
 If $M=G/H$ is a homogeneous strictly nearly Kähler manifold of dimension six, then $M$ is one of the following: 
 $S^6=G_2/{\SU(3)}$, $S^3\times S^3 = \SU(2)^3/\Delta{\SU(2)}$,
  $\cp=\Sp(2)/{\U(1)\times \Sp(1)}$ or $\mathbb{F}=\SU(3)/{\T^2}$.
 \end{prop}
 In each case, the identity component of the group of nearly Kähler automorphisms is equal to $G$, see \cite{davila2012homogeneous}. There are infinitely many freely-acting finite subgroups of the automorphism group of the homogeneous nearly Kähler $S^3\times S^3$, cf. \cite{cortes2015locally}. \par
 In addition, there are two known examples of compact, simply-connected nearly Kähler manifolds which are not homogeneous. They were constructed by Foscolo and Haskins via cohomogeneity one actions on $S^3\times S^3$ and $S^6$ \cite{foscolo2015new}. 
\par 
A 3-dimensional submanifold $L$ of a nearly Kähler manifold is called Lagrangian if $\omega|_L=0$. For the general set-up, we will also work with the more flexible notion of an immersed Lagrangian submanifold, i.e. we  a smooth immersion $\iota\colon L \to M$ such that $\iota^*\omega=0$. However, all examples we encounter are embedded submanifolds.
\par
Because of the nearly Kähler identity $\dd \omega=3\reP$, Lagrangian submanifolds are automatically special Lagrangian. Special Lagrangians in nearly Kähler geometry share some important general properties with special Lagrangians in Calabi-Yau manifolds.
Every special Lagrangian $L$ in $M$ is minimal and orientable, see for example \cite{van2019lagrangian}.

Let $\sff_M$ be second fundamental form of $L$ in $M$.
 Then the cubic form $C(X,Y,Z)=\omega (\sff_M(X,Y),Z)$
 is fully-symmetric, i.e. an element of $\Gamma(S^3(T^*L))$, and traceless when contracted in any two components, see \cite{schafer2010decomposition}.
The cubic form $C$ is also called the fundamental cubic of $L$ and takes values in the intrinsic bundle $\Gamma(S^3(T^*L))$. This means that in order to study special Lagrangians where $C$ satisfies special properties one does not need to specify the normal bundle of $L$. One such special property would be that $C$, or equivalently the second fundamental form, is a parallel section. However, it turns out that this assumption is rather restrictive. Any such Lagrangian is automatically totally geodesic \cite[Theorem 1.1]{zhang2016lagrangian}.\par 
Another special property of $C$ is that it admits symmetries. This approach has been developed in \cite{bryant2006second} for special Lagrangian submanifolds of $\C^3$. By picking a frame in a point $x\in L$ one regards $C$ as a harmonic polynomial of degree three in three variables, i.e. an element of $\mathcal{H}_3(\R^3)$, which is a seven-dimensional vector space. The space $\mathcal{H}_3(\R^3)$ as an $\SO(3)$ module and a generic element in $\mathcal{H}_3(\R^3)$ does not have any symmetries in $\SO(3)$. The possible symmetry groups are classified in \cite[Proposition 1]{bryant2006second}. 
The classification gives a natural ansatz for finding special Lagrangian submanifolds. Impose one of the pointwise symmetries above to every point in $L$. This ansatz has led to the construction of new special Lagrangians in the Calabi-Yau $\C^3$ in \cite{bryant2006second} and in the nearly Kähler $S^6$ \cite{vrancken2003special,lotay2011ruled}. For $\CP^3$ however, this ansatz is less fruitful since the curvature tensor is more complicated so we do not have $\SO(3)$-freedom to change frames as we will see later. However, this framework gives us a way to categorise examples of special Lagrangian that are constructed in different ways. \par
The following result is known for Calabi-Yau manifolds and the nearly Kähler $S^6$ but it holds for any nearly Kähler manifold.
\begin{prop}
\label{prop: Cartan test SL}
Every real analytic surface on which $\omega$ vanishes can locally be uniquely thickened to a special Lagrangian submanifold in $M$.
 Special Lagrangian submanifolds in a nearly Kähler manifold locally depend on two functions of two variables.
\end{prop}
\begin{proof}
 See \cite{lotay2011ruled}, the proof is based on the fact that the Cartan test holds and thus holds for any $\SU(3)$ structure.
\end{proof}
Infinitesimal deformations of nearly Kähler manifolds correspond to eigensections of a rotation operator on $L$ \cite{kawai2014deformations}. It is shown in \cite{van2019lagrangian}, that the moduli space of smooth
Lagrangian deformations of special Lagrangians is a finite dimensional analytic variety. All formally unobstructed infinitesimal deformations are smoothly unobstructed.
\subsection{The nearly Kähler Structure on $\CP^3$}
\label{sec: nk structure cp3}
The nearly K{\"a}hler structure on $\cp$ can be defined through the twistor fibration $\cp \to S^4$. The fibres are projective lines and totally geodesic for the Kähler structure on $\cp$. Since $\cp$ is a sphere bundle inside $\Lambda^2_-(S^4)$ the twistor fibration has a natural connection $T\cp=\mathcal{H}\oplus \mathcal{V}$. The nearly Kähler structure on $\cp$ is defined via the Kähler structure by squashing the metric and reversing the almost complex structure on the vertical fibres. \par
For explicit computations it is convenient to define the nearly K{\"a}hler structure from the homogeneous space structure $\cp=\mathrm{Sp}(2)/{S^1\times S^3}$.
Identify $\mathbb{H}^2$ with $\C^4$ via $\mathbb{H}=\C\oplus j \C$. This identification gives an action of $\mathrm{Sp}(2)$ on $\C^4$ which descends to $\cp$ and acts transitively on that space. The stabiliser of the element $(1,0,0,0)\in \C^4$ is 
\[\left\{ \begin{pmatrix} z & 0 \\ 0 & q \end{pmatrix} \mid z \in S^1\subset \C, \quad q \in S^3\subset \mathbb{H}\right\}\]
which shows $\cp=\mathrm{Sp}(2)/{S^1\times S^3}$. 
Following \cite{xu2010pseudo}, consider the Maurer-Cartan form on $\mathrm{Sp}(2)$ which can be written in components as
\begin{align}
\label{MCeq}
\Omega_{MC}=\begin{pmatrix} i\rho_1 +j\overline{\omega_3}& -\frac{\overline{\omega_1}}{\sqrt{2}}+j\frac{\omega_2}{\sqrt{2}}\\ \frac{\omega_1}{\sqrt{2}}+j\frac{\omega_2}{\sqrt{2}} & i\rho_2+j\tau \end{pmatrix}. 
\end{align}
Since $\Omega_{MC}$ has values in $\mathfrak{sp}(2)$, the one-forms $\omega_1,\omega_2,\omega_3$ and $\tau$ are complex-valued and $\rho_1, \rho_2$ are real-valued. 
The equation $\dd \Omega_{MC}+[\Omega_{MC},\Omega_{MC}]=0$
implies the torsion identity
\begin{align}
\label{firststructure}
 \mathrm{d}\begin{pmatrix} \omega_1 \\ \omega_2 \\ \omega_3 \end{pmatrix}=-\underbrace{\begin{pmatrix} i(\rho_2-\rho_1) &-\overline{\tau} & 0 \\ \tau & -i(\rho_1+\rho_2) & 0 \\ 0 & 0 & 2i\rho_1 \end{pmatrix}}_{A_\omega:=} \wedge\begin{pmatrix} \omega_1 \\ \omega_2 \\ \omega_3 \end{pmatrix} +\begin{pmatrix} \overline{\omega_2}\wedge \overline{\omega_3} \\ \overline{\omega_3}\wedge \overline{\omega_1} \\ \overline{\omega_1}\wedge \overline{\omega_2} \end{pmatrix}.                                                                                                                                                                                                                
\end{align}
and the curvature formula
\[\dd A_\omega = - A_ {\omega}\wedge A_{\omega} + \begin{pmatrix} \omega_1\wedge \overline{\omega_1}-\omega_3\wedge\overline{\omega_3} & \omega_1\wedge \overline{\omega_2} & 0 \\
                                                 \omega_2\wedge \overline{\omega_1} & \omega_2\wedge \overline{\omega_2}-\omega_3\wedge \overline{\omega_3} & 0 \\
                                                 0 & 0 & -\omega_1\wedge \overline{\omega_1}-\omega_2\wedge \overline{\omega_2}+2\omega_3\wedge \overline{\omega_3}
                                                 \end{pmatrix}.
\]
The nearly K{\"a}hler structure on $\cp$ is defined by declaring the forms $s^*\omega_1,s^*\omega_2$ and $s^*\omega_3$ to be unitary $(1,0)$ forms for any local section $s$ of the bundle $\Sp(2)\to \cp$. The resulting almost complex structure and metric do not depend on the choice of $s$. The nearly Kähler forms $\omega,\psi$ are pullbacks of
\begin{align*}
 \frac{i}{2}\sum_{i=1}^3 \omega_i\wedge \overline{\omega}_i,\text{ and } -i \omega_1 \wedge \omega_2\wedge \omega_3,
\end{align*}
respectively.
In general, we will treat the nearly Kähler forms as basic forms on $\mathrm{Sp}(2)$. However, Killing vector fields typically have a simple expression in local coordinates. To contract the nearly Kähler forms on $\cp$ with Killing vector fields we pull back the local unitary $(1,0)$ forms $\omega_1,\omega_2,\omega_3$ on the chart $\mathbb{A}_0=\{Z_0\neq 0\}$ with the local section 
  \[s\colon \mathbb{A}_0\to \Sp(2),\quad  (1,Z_1,Z_2,Z_3) \mapsto \begin{pmatrix} h_1 |Z|^{-1} & -\overline{h}_1^{-1} \overline{h}_2 a \\ h_2 |Z|^{-1} & a \end{pmatrix}.\] Here, 
  \[|Z|^2=1+|Z_1|^2+|Z_2|^2+|Z_3|^2,\quad  h_1=1+j Z_1,\quad  h_2=Z_2+ j Z_3, \quad a=(1+\frac{|h_2|^2}{|h_1|^2})^{-1/2}.\]
This gives the following expressions for the pull-backs
\begin{align}
\label{eq: pull back local frame}
\begin{split}
 s^*\omega_1&=\sqrt{2}|Z|^{-2}( (\overline{Z_3}-\overline{Z_1}Z_2)\dd Z_1+ (1+|Z_1|^2)\dd Z_2) \\
 s^*\omega_2&=\sqrt{2}|Z|^{-2}((-\overline{Z_2}-\overline{Z_1}Z_3)\dd Z_1+(1+|Z_1|^2)\dd Z_3 ) \\
 s^*\omega_3&= |Z|^{-2} (\dd \overline{ Z_1}- \overline{Z_3} \dd \overline{ Z_2}+ \overline{Z_2} \dd \overline{ Z_3}).
 \end{split}
\end{align}
To show these formulae, note that the pullback of the Maurer-Cartan form via $s$ is 
\[s^*(\Omega_{MC})=\begin{pmatrix} \overline {h}_1 |Z|^{-1} & \overline{h_2} |Z|^{-1} \\ -h_2 h_1^{-1} a & a 
  \end{pmatrix} \begin{pmatrix} \mathrm{d}( |Z|^{-1} h_1) & d(-\overline{h}_1^{-1} \overline{h}_2 a)\\ \mathrm{d} ( |Z|^{-1} h_2) & \mathrm{d}a \end{pmatrix}.\]                                                                                                         
Combining this with \cref{MCeq} yields
\begin{align*}
 (i s^*\rho_1 +j s^* \overline{\omega_3})&=|Z|^{-2} (\overline{h_1} \mathrm{d}h_1 + \overline{h_2} \mathrm{d}h_2)+R \\
 \frac{1}{\sqrt{2}a |Z|}(s^*\omega_1+j s^*\omega_2)&=-h_2 h_1^{-1} \mathrm{d}h_1 + \mathrm{d}h_2,
 \end{align*}
where $R$ is a real term. Equations \ref{eq: pull back local frame} follow by splitting the quaternionic-valued differential forms on the right-hand side into their $\C$ and $j\C$ part.
\section{Structure Equations for Special Lagrangians}
\label{sec: structure eqns in nk manifolds}
The structure equations for a special Lagrangian manifold in Calabi-Yau $\C^3$ were established in \cite{bryant2006second} and for nearly Kähler $S^6$ in \cite{lotay2011ruled}. We generalise the equations  to the setting of a general nearly Kähler manifold. The main difference is the appearance of an extra curvature term. We characterise nearly Kähler manifolds by differential identities on the frame bundle, as done in \cite{bryant2006geometry}. If an index appears on the right-hand side but not on the left-hand side of an equation, summation over the index set $\{1,2,3\}$ is implicit.\par
Let $M^6$ be a nearly Kähler manifold and consider the $\SU(3)$-frame bundle $P_{\SU(3)}$. Let $(\zeta_1,\zeta_2,\zeta_3)\in \Omega^1(P_{\SU(3)},\C^3)$ be the tautological one-forms on $P_{\SU(3)}$ and let $\phi\in \Omega^1(P,\mathfrak{su}(3))$ be the nearly Kähler connection one-form on $P_{\SU(3)}$, giving the torsion relation
\begin{align}
\label{eqn: torsion-relation}
\dd \begin{pmatrix} \zeta_1 \\ \zeta_2 \\ \zeta_3 \end{pmatrix} = -\phi \wedge \begin{pmatrix} \zeta_1 \\ \zeta_2 \\ \zeta_3 \end{pmatrix} +\begin{pmatrix} \overline{\zeta}_2 \wedge \overline{\zeta}_3 \\ \overline{\zeta}_3 \wedge \overline{\zeta}_1 \\ \overline{\zeta}_1 \wedge \overline{\zeta}_2 \end{pmatrix} 
\end{align}
and the curvature identity
\begin{align}
\label{eqn: curvature-relation}
\dd \phi_{ij}=-\phi_{ik}\wedge \phi_{kj}+K_{ijpq} \zeta_{q}\wedge \overline{\zeta_p}.
\end{align}
In particular, the curvature of $\nk$ is always of type $(1,1)$.
In \cite{bryant2006geometry} it is remarked that the tensor $K$ can be written as sum 
\[K_{ijpq}=K'_{ijpq}+\frac{3}{4}\delta_{pi}\delta_{qj}-\frac{1}{4}\delta_{ij}\delta_{pq}\]
where $K'$ has the following symmetries
\[K'_{ijpq}=K'_{pjiq}=K'_{iqpj}=\overline{K'_{jiqp}}, \text{ and } \sum_i K'_{iipq}=0.\]
The tensor $K'$ vanishes exactly when $M$ is the round six-sphere.
The nearly Kähler forms are expressed in terms of $\zeta_i$ by 
\begin{align}
\label{eq: nk forms from co-frame}
 \omega=\frac{i}{2}\sum_i \zeta_i \wedge \overline{\zeta_i}, \quad \psi=-i \zeta_1\wedge\zeta_2\wedge \zeta_3.
\end{align}
Note the difference from \cite{bryant2006geometry} in the convention for $\psi$ in order to satisfy the standard nearly Kähler integrability equations.\par
The torsion-relation \cref{eqn: torsion-relation} and curvature-relation \cref{eqn: curvature-relation} yield differential identities for the connection one-form and tautological one-form on the frame bundle $P_{\SU(3)}$. If $L$ is a special Lagrangian submanifold in $M$ then one obtains more differential identities because the frame bundle $P_{\SU(3)}$ admits a natural reduction to an $\SO(3)$ bundle over $L$. The reason for this is that, at the tangent level, a Lagrangian subspace looks like $\R^3$ in $\C^3$, which defines the restriction
\[P_{\SO(3)}=\{p\colon \C^3\to TM,\quad p \in  P_{\SU(3)}|_L \mid  p(\R^3)=TL\}.\]
If $\dd z_1,\dd z_2,\dd z_3$ are the standard complex-valued one-forms on $\C^3$ then $\R^3\subset \C^3$ is characterised as the 3-dimensional subspace of $\C^3$ on which the imaginary parts of $\dd z_i$ vanish. Similarly, our aim is to describe the reduction $P_{\SO(3)}$ as the vanishing set of one forms on $P_{\SU(3)}$. To that end, split the forms $\zeta_i= \sigma_i+i \eta_i$ and $\phi=\alpha+i\beta$ into real and imaginary part. The bundle $P_{\SO(3)}$ is now defined by imposing the condition $\eta_i=0$. \par
This characterisation implies more differential identities.
From the torsion-relation we get
\begin{align*}
 \dd \sigma_i&=-\alpha_{ij} \wedge \sigma_j +\beta_{ij}\wedge \eta_j+ \sigma_k \wedge \sigma_l-\eta_k\wedge \eta_l \\
 \dd \eta_i &=-\beta_{ij} \wedge \sigma_j-\alpha_{ij}\wedge \eta_j -\sigma_k\wedge \eta_l -\eta_k \wedge \sigma_l
\end{align*}
where $(i, k, l)$ is an cyclic permutation of $(1, 2, 3)$.
The condition $\eta_i=0$ implies $\beta_{ij}\wedge \sigma_j=0$. By Cartan's lemma, we have $\beta_{ij}=h_{ijk} \sigma_k$ or $\beta=h \sigma$ where $h$ is a fully symmetric three-tensor. In fact, this tensor corresponds to the fundamental cubic up to a factor, just as in the case of special Lagrangians in $\C^3$ or in $S^6$.
\par
On the reduced bundle, we split $K$ into real and imaginary part, \[K_{ijpq} \zeta_q \wedge \overline{\zeta_p}=K_{ijpq} \sigma_q \wedge \sigma_p=(R_{ijpq}+i S_{ijpq})\sigma_q \wedge \sigma_p=(-R_{ijpq}-i S_{ijpq})\sigma_p \wedge \sigma_q.\]
This also allows us to split the curvature identity into real imaginary part
\begin{align*}
 \dd \alpha_{ij}&=-\alpha_{ik}\wedge \alpha_{kj}+\beta_{ik}\wedge \beta_{kj}-R_{ijpq}\sigma_p \wedge \sigma_q\\
 \dd \beta_{ij}&=-\beta_{ik}\wedge \alpha_{kj}-\alpha_{ik} \wedge \beta_{kj}-S_{ijpq}\sigma_p \wedge \sigma_q.
\end{align*}
To write these equations more invariantly, let \[ [\sigma]=\begin{pmatrix} 0 & \sigma_3 & -\sigma_2 \\ - \sigma_3 & 0 & \sigma_1 \\ \sigma_2 & -\sigma_1& 0 \end{pmatrix}.\]
We can summarise the equations on the reduced bundle over $L$ in tensor notation
\begin{align}
\label{lagstructureeqns1}
\beta\wedge \sigma&=0 \\
\label{lagstructureeqns2}
\dd \sigma&=-\alpha \wedge \sigma -\frac{1}{2}[\sigma]\wedge \sigma\\
\label{lagstructureeqns4}
 \dd \alpha&=- \alpha \wedge \alpha+ \beta\wedge\beta- R \sigma \wedge \sigma \\
  \label{lagstructureeqns3}
 \dd \beta&= -\beta \wedge \alpha-\alpha\wedge \beta -S \sigma \wedge \sigma
\end{align}
where $(\sigma\wedge \sigma)_{ pq}=\sigma_p\wedge \sigma_q$. The matrix of one forms $\beta$ is completely defined by the symmetric tensor $h$. The advantage to work with $h$ is that its components are not one-forms but functions, allowing us to rewrite \cref{lagstructureeqns1}, \cref{lagstructureeqns4} and \cref{lagstructureeqns3}
\begin{align*}
\beta=h \sigma \quad
 \dd \alpha=- \alpha \wedge \alpha+h\sigma \wedge h \sigma+ \frac{3}{4}\sigma\wedge\sigma - R \sigma \wedge \sigma, \quad
   0=(\dd h +((h \alpha+\frac{1}{2} h[\sigma]))+S \sigma) \wedge \sigma.
\end{align*}
The Levi-Civita connection one-form of the induced metric on $L$ is $\alpha+\frac{1}{2}[\sigma]$. Note that the forms $\sigma$ differ by a factor 2 from the orthonormal one forms considered in \cite{lotay2011ruled}.
\par 
If $M=G/H$ is one of the homogeneous nearly Kähler manifolds then a special Lagrangian submanifold can locally be recovered from a solution to \cref{lagstructureeqns1}-\cref{lagstructureeqns4}, which we will make precise now. There is a splitting
$\mathfrak{g}=\mathfrak{h}\oplus \mathfrak{m}$
such that $\mathrm{Ad}_H(\mathfrak{m})\subset \mathfrak{m}$. The nearly Kähler structure then yields an $\mathrm{Ad}(H)$ invariant special unitary basis $\omega_1,\omega_2,\omega_3$ on $\mathfrak{m}\cong \C^3$. Up to a cover, $G$ embeds into the $\SU(3)$-frame bundle $P_{\SU(3)}$ via the adjoint action $H\to \SU(\mathfrak{m})$. 
Under this identification
\[\psi+(\zeta_1,\zeta_2,\zeta_3)\in \mathfrak{h}\oplus \C^3\cong \mathfrak{h}\oplus \mathfrak{m}\]
is the Maurer-Cartan form $\omega_G$ on $G$. In other words, the nearly Kähler connection is equal to the canonical homogeneous connection on $G\to M$, see \cite{butruille2010homogeneous}. The following proposition guarantees that for the homogeneous nearly Kähler manifolds we can locally recover the special Lagrangian from a solution of the structure equations. Since $\alpha$ and $\beta$ determine the first and second fundamental form, this can be viewed a Bonnet-type theorem.
\begin{prop}
\label{prop: bonnet}
 Let $M=G/H$ be a homogeneous nearly Kähler manifold, $L^3$ be a simply-connected three manifold and $\sigma\in \Omega^1(L,\R^3)$, defining a linearly independent co-frame at each point, $\alpha\in \Omega^1(L,\mathfrak{so}(3))$ and $\beta \in \Omega^1(L,S^2(\R^3))$ satisfying the equations \ref{lagstructureeqns1}-\ref{lagstructureeqns3}. Then there is a special Lagrangian immersion $L\to M$, unique up to isometries, with $\alpha,\beta$ determining the metric and second fundamental form of $L$ in $M$.
\end{prop}
\begin{proof}
 Define the form $\gamma=\alpha+i \beta+(\sigma_1,\sigma_2,\sigma_3)\in \mathfrak{h}\oplus \mathfrak{m}\cong \mathfrak{g}$. Since $\sigma, \alpha,\beta$ satisfy the equations \ref{lagstructureeqns1}-\ref{lagstructureeqns3} we have 
 $\dd \gamma+[\gamma,\gamma]=0$. The statement now follows from Cartan's theorem, just as the classical Bonnet theorem for surfaces in $\R^3$.
\end{proof}
\begin{rmk}
\label{remark: local frames}
 Note that the tautological one form $(\zeta_1,\zeta_2,\zeta_3)$ can also be regarded as an element in $\Gamma(P,\mathrm{End}(\C^3))$. With this identification, a local section $s$ of $L\supset U\to P_{\SO(3)}$ gives a section $\Gamma(U,(T^\vee M)|_L\otimes \C^3)\cong \Omega^1(U,\C^3)$. Then $s^*\eta_i$ vanishes on $TL$ while $s^*\sigma$ vanishes on the normal bundle.
\end{rmk}

\section{An Angle Function for Special Lagrangians}
Since twistor fibres are $J$-holomorphic they can never be contained in a special Lagrangian submanifold. Generically, a special Lagrangian intersects every twistor fibre transversally. However, there is a special class of special Lagrangians which are circle bundles over superminimal surfaces in $S^4$. We review this construction and define an angle function $L\to [0,\frac{\pi}{4}]$ which has value $\frac{\pi}{4}$ if $L$ intersects a twistor fibre non-transversally. We use a gauge transformation, which depends on $\theta$, to use the moving frame setup from the previous section for special Lagrangians in $\CP^3$. We identify special solutions to the resulting structure equations, all of which turn out to be homogeneous.
\label{sec: angle function SL}
\subsection{The Linear Model}
We start with the study of Lagrangian subspaces in a twistor space on the tangent level.
The space of special Lagrangian subspaces of $\C^n$ is identified with the homogeneous space $\SU(n)/{\SO(n)}$.
Twistor nearly Kähler spaces have the property that the holonomy of the nearly Kähler connection reduces to $\mathrm{U}(2)\cong \mathrm{S}(\mathrm{U}(2)\times \mathrm{U}(1))$. The two-form splits into a horizontal and vertical part $\omega=\omega_{\mathcal{H}}+\omega_{\mathcal{V}}$. 
So, in order to understand how frames can be adapted further to a special Lagrangian of a twistor space, we study the linear problem first.\par
Let $(b_1,b_2,b_3)$ denote the standard basis of $\C^3$ with dual basis $(\omega_1,\omega_2,\omega_3)$ and let $\omega_{\mathcal{H}}=\frac{i}{2} (\omega_1\wedge \bar{\omega}_1+\omega_2\wedge \bar{\omega}_2)$ as well as $\omega_{\mathcal{V}}=\frac{i}{2} (\omega_3\wedge \bar{\omega}_3)$. Let $H\cong \mathrm{S}(\mathrm{U}(2)\times \mathrm{U}(1))$ be the stabiliser of $\omega_{\mathcal{V}}$ inside $\SU(3)$. Let also $\psi=\reP+i\imP$ be the complex-valued three form $-i\omega_1\wedge\omega_2\wedge\omega_3$ on $\C^3$.
We have abused notation slightly here, since $\omega,\psi$ are forms on the nearly Kähler manifold but also denote their linear models on $\C^3$.\par
For a complex subspace $W\subset \C^3$ denote by $\mathrm{SLag}(W)$ the set of all special Lagrangian subspaces of $W$. 
 By $\C^2\subset \C^3$ we refer to the subspace spanned by $b_2$ and $b_3$. Note that $\mathrm{SLag}(\C^2)\cong S^2$ and that $\mathrm{U}(1)\subset \SU(2)$ acts from the left on this space. The quotient is an interval and the following lemma gives a description of each representative.
\begin{lma}
\label{representativelemma}
 Under the action of $\mathrm{U}(1)\cong\{\mathrm{diag}(e^{i\varphi},e^{-i\varphi})\}\subset \SU(2)$ any element in $\mathrm{SLag}(\C^2)$ has a unique representative of the form $V_{\theta}=\mathrm{span}(-i e^{-i\theta} b_2- e^{-i\theta} b_3, e^{i\theta} b_2 + i e^{i\theta} b_3)$, for $0\leq \theta \leq \pi/2$.
\end{lma}
\begin{proof}
 Special Lagrangian planes in $\C^2$ are parametrised by $\SU(2)/{\SO(2)}$. Thus, we have to find a unique representative of the action $(A,B) X=A X B^{-1}$ of $K=\mathrm{U}(1)\times \SO(2)$ on $\SU(2)$, which is the action of a maximal torus in $\SO(4)$ acting on $S^3$. The standard torus $\mathrm{U}(1)\times \mathrm{U}(1)\subset \mathrm{U}(2)\subset \SO(4)$ acting on $S^3$ admits unique representatives of the form $(\cos(\theta),0,\sin(\theta),0)$ for $0\leq \theta \leq \pi/2$. The statement follows by conjugating the action of $K$ to the standard torus action.
\end{proof}
For any subspace $W\subset \C^3$ denote by $K_W$ the kernel of the projection onto $\mathrm{span}(b_3)$ and by $n_W$ its dimension. Let 
\begin{align}
\label{eq: transformation matrix theta}
T_{\theta}=\left(
\begin{array}{ccc}
 1 & 0 & 0 \\
 0 & -\frac{i e^{-i \theta }}{\sqrt{2}} & \frac{e^{i \theta }}{\sqrt{2}} \\
 0 & -\frac{e^{-i \theta }}{\sqrt{2}} & \frac{i e^{i \theta }}{\sqrt{2}} \\
\end{array}
\right)
\end{align}
and $W_{\theta}$ be the image of $T_{\theta}$ when applied to the standard $\R^3$ in $\C^3$, i.e.
$W_{\theta}=\mathrm{span}(b_1,-i e^{-i\theta} b_2- e^{-i\theta} b_3, e^{i\theta} b_2 + i e^{i\theta} b_3)$.
\begin{prop}
\label{prop: representatives Lag subspaces}
 Any special Lagrangian subspace $W\subset \C^3$ admits a unique representative $W_{\theta}$ for $0\leq \theta \leq \pi/4$, under the action of $H$. Furthermore, $n_W=2$ if and only if $\theta=\pi/4$.
\end{prop}
\begin{proof}
Since $W$ is Lagrangian, $n_W\geq 1$. If $n_W=2$ then $W$ is represented by the standard $\R^3$ in $\C^3$ and $n_{W_{\theta}}=2$ if and only if $\theta=\pi/4$. So from now on we assume that $n_W=1$.
Consider the map $l \colon\mathrm{Gr}_2(\C^2)\to \mathrm{Gr}_3(\C^3), \quad V\mapsto \mathrm{span}(b_1,V)$. Note that $W_{\theta}=l(V_{\theta})$ and that $l$ descends to a map $\hat{l}\colon \mathrm{Lag}(\C^2)/{\mathrm{U}(1)}\to \mathrm{Lag}(\C^3)/H.$ To show surjectivity observe that for $W\in \mathrm{Lag}(\C^3)$ we have $K_W\subset \mathrm{span}(b_1,b_2)$. So by acting with $H$ we can achieve that $K_W$ is spanned by $b_1$. Furthermore, observe that  
\[\begin{pmatrix}-1 & 0&0 \\ 0&-1&0 \\ 0& 0 &1 \end{pmatrix}T_{\theta} \begin{pmatrix}-1 & 0&0 \\ 0&0&1 \\ 0& 1 &0 \end{pmatrix}=T_{\pi/2-\theta} \] 
which means that $W_{\theta}=W_{\theta'}$ for $\theta+\theta'=\pi/2$.
We have shown that any element in $\mathrm{Lag}(\C^3)$ is represented by a $W_{\theta}$ for $0\leq \theta\leq \pi/4$. The uniqueness follows from the observation that $\omega_{\mathcal{V}}$ has norm $\frac{1}{2}|\mathrm{cos}(2\theta)|$ when restricted to the vector space $W_{\theta}$.
\end{proof}
If $w_1,w_2,w_3$ is a basis of $W$ such that $w_1\in K_W$ and $w_2,w_2\in K_W^\perp$ then $\theta$ can be computed by the formula
\begin{align}
\label{eq: formula for theta}
 \frac{1}{2\|w_1\|}|\mathrm{cos}(2\theta)|\psi^{-}(w_1,w_2,w_3)=\omega_{\mathcal{V}}(w_2,w_3).
\end{align}
\Cref{prop: representatives Lag subspaces} gives a geometric interpretation of the boundary value $\pi/4$. In \cref{prop: CR submanifolds} we relate the case $\theta=0$ to CR-manifolds in the Kähler $\cp$, so we study this case on the linear level first.
Motivated by the existence of the two almost complex structures $J_1$ and $J_2$ on the twistor space, consider the almost complex structure \begin{align}
\label{def: J'}
J'\colon (b_1,b_2,b_3)\mapsto (i b_1,i b_2, -i b_3)                                                                                                                                                                                                                                                                                                                                                                                                                        
\end{align}
on $\C^3$. 
Any special Lagrangian subspace $W$ in $\C^3$ splits as $K_W\oplus K_W^\perp$.
\begin{lma}
\label{lma: CR subspaces}
 If $\theta\neq \frac{\pi}{4}$ then $J(K_W)$ is orthogonal to $W$. The subspace $K_W^\perp$ is invariant under $J'$ if and only if $\theta=0$.
\end{lma}
\begin{proof}
 The endomorphism $J'$ commutes with the action of $H$ on $\C^3$, so it suffices to prove the statement for $W_\theta$. If $\theta=\frac{\pi}{4}$ then $K_W^\perp$ is one-dimensional so it cannot be invariant under $J'$. Otherwise, $K_W$ is spanned by $b_1$ and $K_W^\perp$ equals $V_\theta$. Clearly $Jb_1=ib_1$ is orthogonal to $W$. The statement follows by observing that $V_\theta$ is invariant under the endomorphism $(b_2,b_3)\mapsto (ib_2,-ib_3)$ if and only if $\theta=0$.
\end{proof}
The following lemma can be proven by standard computations in $\SU(3)$ and is important for adapting frames on special Lagrangians in twistor spaces.
\begin{lma}
\label{stabiliser}
 Let $H_{\theta}=T_{\theta}^{-1}H T_{\theta} \cap \SO(3)$ be the stabiliser group of $W_{\theta}$ in $H$ with Lie algebra $\mathfrak{h}_{\theta}$. Then
\begin{align*}\mathfrak{h}_{\theta}=\begin{cases} \R \cdot \begin{pmatrix} 0 & 1 & -1\\ -1 & 0 & 0 \\ 1 & 0 & 0 \end{pmatrix} & \theta=\pi/4 \\
\{0\} \oplus \mathfrak{so}(2) & \theta=0\\
\{0\} & \text{otherwise}
\end{cases}
\qquad
T_{\theta} \mathfrak{h}_{\theta} T_{\theta}^{-1}=\begin{cases} \R \cdot  \begin{pmatrix} 0 & -1+i & 0\\ 1+i & 0 & 0 \\ 0 & 0 & 0 \end{pmatrix} & \theta=\pi/4 \\
\{0\} \oplus \mathfrak{s(u(1)\oplus u(1))} & \theta=0 \\
\{0\} & \text{otherwise}
\end{cases}
\end{align*}
Then $H_{\theta}$ is generated by $\exp(\mathfrak{h}_{\theta})$ and the element $\mathrm{diag}(1,-1,-1)$. In particular, $H_{\theta}$ is isomorphic to $\mathrm{O}(2)$ if $\theta=\pi/4$, to $\SO(2)$ if $\theta=0$ and to $\mathbb{Z}_2$ otherwise.
\end{lma}
The action of $H$ on $\mathrm{Lag}(\C^3)$ is a smooth cohomogeneity one action. The orbit at $W_{\theta}$ is diffeomorphic to $H/({T_{\theta}H_{\theta} T_{\theta}^{-1}})$
and is singular for $\theta=0,\pi/4$ and of principal type otherwise.
The principal orbits are diffeomorphic to $H/{\langle \mathrm{diag}(1,-1,-1)\rangle}\cong \mathrm{U}(2)/{\langle \mathrm{diag}(1,-1)\rangle}$. The orbit of $W_0$ is diffeomorphic to $H/{(\{1\}\times S(\mathrm{U}(1)\times \mathrm{U}(1)))}\cong \mathrm{U}(2)/{(\{1\}\times \mathrm{U}(1))}\cong S^3$. 
Observe that ${T_{\pi/4}H_{\pi/4} T_{\pi/4}^{-1}}$ is conjugated to the $\mathrm{O}(2)$ subgroup generated by 
\[S(\U(1)\times \U(1)) \text{ and } \begin{pmatrix} 0 & -1 \\ -1&0 \end{pmatrix}.\]
 This subgroup is equal to the preimage of $[([1,0],1)$ of the map \[\mathrm{U}(2)\to (\mathbb{CP}^1\times S^1)/{\mathbb{Z}_2}, \quad A\mapsto [[A (1,0)^T],\mathrm{det}(A)].\] Here $\mathbb{Z}_2$ acts as the antipodal map on both $\mathbb{CP}^1\cong S^2$ and on $S^1$. Hence, the orbit of $W_{\pi/4}$ is diffeomorphic to $(S^2\times S^1)/\mathbb{Z}_2$.\par
 The following lemma summarises these observations.
 \begin{lma}
  The action of $H$ on $\mathrm{Lag}(\C^3)$ is of cohomogeneity one. The principal orbit is diffeomorphic to $\mathrm{U}(2)/{\mathbb{Z}_2}$, two singular orbits occur at $\theta=0$ and $\theta=\frac{\pi}{4}$. The orbit $W_0$ is diffeomorphic to $S^3$ and that of $W_{\pi/4}$ to $(S^2\times S^1)/\mathbb{Z}_2$.
 \end{lma}
\subsection{Adapting Frames}
\label{sec: adapting frames}
We now assume that $M$ is a nearly Kähler twistor space over a Riemannian four manifold $N$. In other words, $M$ is either $\CP^3$ or the flag manifold $\mathbb{F}$. Lagrangian submanifolds of the latter have been studied in \cite{storm2020lagrangian} so our interest is in $\CP^3$ in this chapter. Before using the explicit description of $\CP^3$ we give a few general statements that could be useful for generalisations to other spaces, such as non-nearly Kähler twistor spaces.\par 
Given a special Lagrangian submanifold $L \subset M$, we clearly have $T_xL \in \mathrm{Lag}(T_xM)$ for $x \in L$. Since the frame bundle reduces to $H$ there is a map $\mathrm{Lag}(TM|_L) \to \mathrm{Lag}(\C^3)/H$. Hence, $\theta$ can be understood as a map from $L$ to the interval $[0,\frac{\pi}{4}]$ and $T_{\theta}$ from $L$ to $\SU(3)$. 
We now apply our knowledge of the action of $H$ on $\mathrm{Lag}(\C^3)$ to obtain a further frame reduction for special Lagrangian submanifolds in nearly Kähler twistor spaces.
In that case, the holonomy of the nearly Kähler connection on $M$ reduces to $H$, so $P_{\SU(3)}$ reduces to an $H$-bundle and we can assume $\phi_{13}=\phi_{23}=\phi_{31}=\phi_{32}=0$. This means that there are two different reductions of $P|_L$:
The first is to an $H$-bundle $P_H=\{p\colon \C^3\to TM, p \in P_{\SU(3)}|_L \mid p(b_3)\in \mathcal{V}\}$, simply because $P_{\SU(3)}$ itself reduces to an $H$ bundle. The second reduction is to an $\SO(3)$-bundle $P_{\SO(3)}=\{p\colon \C^3\to TM, p \in  P_{\SU(3)}|_L \mid  p(\R^3)=TL \}$ or equivalently by imposing $\eta_i=0$ as in \cref{sec: structure eqns in nk manifolds}.
\par
If $TL\cap \mathcal{V}$ is a rank one bundle, or equivalently $\theta\equiv \frac{\pi}{4}$, then the intersection $P_{\SO(3)}\cap P_H$ is a $H\cap \SO(3)$ bundle. We will derive its structure equations in \cref{subsection: pi4-Lagrangian}.
If $\theta$ avoids the value $\frac{\pi}{4}$ then the intersection $TL\cap \mathcal{V}$ is trivial and $P_{\SO(3)}\cap P_H=\emptyset$ which
precludes the existence of a distinguished frame. However, by \cref{stabiliser} we can apply a gauge transformation to guarantee a non-empty intersection. 
\par For $x\in L$ there is a frame in $P_H$ which maps $W_{\theta}$ to $TL$. Such a frame is unique up to the action of the stabiliser of $W_{\theta}$ in $H$, which is computed in \cref{stabiliser}. This means that 
\begin{align} 
\label{eqn: reduced bundle SL}
Q=P_H T_{\theta}\cap P_{\SO(3)}\neq \emptyset. \end{align}
This is a principal bundle over $L$ with structure group given as in \cref{stabiliser} if $\theta$ is either equal to $0$ or $\frac{\pi}{4}$ everywhere or if $\theta$ avoids these values altogether. In the latter case, the structure group is discrete. 
  We first describe all special Lagrangians where $\theta$ is constant and equal to one of the boundary values everywhere. If $\theta\equiv \frac{\pi}{4}$ then $L$ intersects every twistor fibre in a circle and maps to a surface in $N$.
\subsection{Lagrangians with $\theta\equiv \frac{\pi}{4}$}
\label{subsection: pi4-Lagrangian}
There is a general construction for Lagrangian submanifolds in the twistor space $Z$ of an arbitrary Riemannian four-manifold $N$ due to Storm \cite{storm2020note} and Konstantinov \cite{konstantinov2017higher}. 
To make sense of how a Lagrangian submanifold in $Z$ is defined, recall that $Z$ carries two almost complex structures $J_1, J_2$ and metrics $g_{\lambda}$ for $\lambda \in \R_{\geq 0}$.
For a surface $X\subset N$ define the circle bundle $L_X\subset Z(N)$ with fibre over $x\in X$ equal to $\{J \in Z_x(N)\mid J(T_xX)=\nu_x\}$. Geometrically, the fibre of $L_X$ at $x\in X$ is the equator in each twistor fibre, which is diffeomorphic to $S^2$, relative to the twistor lift of $X$ at $x$. It turns out that this construction gives a lot of examples of Lagrangians in twistor spaces.
\begin{prop}\cite{storm2020note}
\label{prop: storm Lagrangian}
 The submanifold $L_X$ is Lagrangian in $Z$ for both $J_1$ and $J_2$ and every $g_{\lambda}$ if $X$ is superminimal. Conversely, if $L_X$ is Lagrangian for any $J_a$ and $g_{\lambda}$, then $X$ is superminimal.
\end{prop}
Assume $L$ is Lagrangian with $\theta\equiv \frac{\pi}{4}$ so $TL\cap \mathcal{H}$ and $TL\cap \mathcal{V}$ are a rank two and a rank one bundle and 
$TL=TL\cap \mathcal{H} \oplus TL\cap \mathcal{V}.$
So $L$ is also Lagrangian for $J_1$ and $L$ arises via the construction above.
In this case the intersection $P_{\mathrm{O}(2)}=P_H\cap P_{\SO(3)}$ is an $S(\mathrm{O}(2)\times O(1))$ bundle which is defined by imposing $\eta_i=0$ for $i=1,\dots,3$ on $P_H$. 
Since $\beta_{32}=\beta_{23}=\beta_{31}=\beta_{13}=0$ the equation $\beta\wedge \sigma=0$ implies that $\beta_{33}$ lies in the span of $\sigma_3$ and $\beta_{11},\beta_{22}$ lie in the span of $\sigma_1$ and $\sigma_2$. Since $\mathrm{Tr}(\phi)=0$ this implies that $\beta_{33}=0=\beta_{11}+\beta_{22}$, i.e. $\phi$ takes values in $\mathfrak{su}(2)$ when restricted to $P_{\mathrm{O}(2)}$. 
\par
We can view $(\sigma_1,\sigma_2,\sigma_3,\eta_1,\eta_2,\eta_3)$ locally as an orthonormal co-frame on $TM|_L$, see \cref{remark: local frames}. The forms $\sigma_i$ vanish on the normal bundle while $\eta_i$ vanish on $TL$. The form $\sigma_3$ is dual to the unit vector field tangent along the fibres of $L\to X$. Since $\beta_{33}=0$ this means that the fibres of $L\to X$ are in fact geodesics.
Since twistor fibres are totally geodesic $\mathbb{CP}^1\subset M$ these geodesics are great circles in the twistor fibres. 
\par 
Since $\beta_{3i}=0$ for $i=1,2,3$ this implies $h_{3ij}=0$ so the fundamental cubic is of the form 
\[a (x_1^3 - 3 x_1 x_2^2) + b (x_2^3 - 3 x_2 x_1^2).\]
We have therefore shown.
\begin{prop}
 The fundamental cubic of $L_X$ in a nearly Kähler twistor space either vanishes or has stabiliser $S_3$.
\end{prop}
We can also recover the result that $X$ is superminimal by showing that the second fundamental form $X$ in $N$ is complex-linear and using \cite[Proposition 1c]{montiel1997second}.
\begin{rmk}
 Bryant considers special Lagrangians of the form $C^1\times \Sigma^2 \subset \C \oplus \C^2=\C^3$. The cubic form of such submanifolds is always stabilised by $S_3$. These examples are somewhat analogous to horizontal Lagrangians whose fundamental cubic also admits an $S_3$ symmetry. 
\end{rmk}
From now on, we will work specifically with $M=\CP^3$.
We have seen that $L_X$ is either totally geodesic or its fundamental cubic has stabiliser $S_3$. If $L_X$ is homogeneous then $X$ is a homogeneous superminimal surface in $S^4$. Such a surface is equal to a totally geodesic $S^2\subset S^4$ or the Veronese curve in $S^4$. 
Hence, there are only two different examples of homogeneous special Lagrangian submanifolds with $\theta\equiv \frac{\pi}{4}$. Both of them are known as Lagrangians for the Kähler structure on $\CP^3$.
\begin{example}[The standard $\mathbb{RP}^3$]
 \label{standardRP3}
The standard $\mathbb{RP}^3\subset \mathbb{CP}^3$ is a totally geodesic special Lagrangian submanifold. It fibres over a totally geodesic $S^2$ in $S^4$ under the twistor fibration.
It is the orbit of $\left\{\begin{pmatrix}a & -\bar{b} \\ b & \bar{a}\end{pmatrix}\mid a,b\in \R \oplus j \R, \quad |a|^2+|b|^2=1\right\}\cong \SU(2)$ on $[1,0,0,0]$.
\end{example}
The second example was discovered in \cite{chiang2004new} and is described in \cite{konstantinov2017higher} in terms of the twistor fibration.
\begin{example}[Chiang Lagrangian]
\label{Chiang-Lagrangian}
 The $\SU(2)$ subgroup of $\mathrm{Sp}(2)$ which comes from the irreducible representation of $\SU(2)$ on $\C^4=S^3(\C^2)$ has a special Lagrangian orbit at $[1,0,0,1]\in \CP^3$.
 This example is known as the Chiang Lagrangian and fibres over the Veronese surface in $S^4$.
  The $\SU(2)$ subgroup acts with stabiliser $S_3$ on $[1,0,0,1]$. The stabiliser subgroup induces the full symmetry group of the fundamental cubic since the Chiang Lagrangian is not totally geodesic. 
\end{example}
Since superminimal curves in $S^4$ have an explicit Weierstraß parametrisation one can produce many (explicit) examples of special Lagrangians in $\CP^3$. However, our focus is on exploring special Lagrangians which do not arise from superminimal surfaces.
\subsection{Changing the Gauge}
If one expresses the nearly Kähler structure on $\CP^3$ in terms of local coordinates one can work out a system of PDE's which, at least locally, describes special Lagrangian submanifolds. However, this approach is not very likely to succeed since local coordinates on $\CP^3$ are not an elegant way to define its nearly Kähler structure. Of more geometric importance are the first and second fundamental form and \cref{prop: bonnet} shows that locally they contain all information about the submanifold. We use a gauge transformation, which depends on the function $\theta$ to describe the structure equations for a special Lagrangian in $\CP^3$. \par
The bundle $\mathrm{Sp}(2)$ embeds into the frame bundle of $\mathbb{CP}^3$ via the adjoint action of $S^1\times S^3$ on $\mathfrak{m}$ which factors through the double cover $S^1\times S^3\to \mathrm{U}(2)$. So, on the level of structure equations we identify $P_H$ with $\mathrm{Sp}(2)$.  We apply the gauge transformation $T_{\theta}$ to $\Sp(2)$, which defines the bundle $Q$ as in \cref{eqn: reduced bundle SL}. This bundle has a reduced structure group, depending on the behaviour of the function $\theta$, which is made precise in \cref{stabiliser}. For example, if $\theta$ avoids the values $0$ and $\frac{\pi}{4}$ then the structure group of $Q$ is $\Z_2$.\par
Recall from \cref{sec: nk structure cp3} that $\Sp(2)$ is an $S^1\times S^3$ principal bundle over $\CP^3$. A local unitary frame for the nearly Kähler structure on $\CP^3$ is obtained by pulling back the forms $(\omega_1,\omega_2,\omega_3)$, which are components of the Maurer-Cartan form on $\Sp(2)$.
We can realise the bundle $Q$ by setting $T^{-1}_{\theta}(\omega_1,\omega_2,\omega_3)=(\zeta_1,\zeta_2,\zeta_3)$, where $T_{\theta}$ is defined in \cref{eq: transformation matrix theta},
and imposing the equations
\begin{align}
\label{reduced bundle}
 \eta_1=0,\quad \eta_2=0,\quad \eta_3=0.
\end{align}
Our aim is to compute the differentials of the one forms $\zeta_i$ and also of $\rho_i$ and $\tau$ on the reduced bundle $Q$. We will achieve this by first computing the connection and curvature form in the transformed frame and then applying \cref{lagstructureeqns1}-\cref{lagstructureeqns3}. We begin by applying the transformation formula for a connection-one form under the gauge transformation $T_{\theta}$
\begin{align}
\label{gauge trafo}
 \phi=T_{\theta}^{-1} A_{\omega} T_{\theta}+T_{\theta}^{-1}d T_{\theta}.
\end{align}
Here $A_\omega$ is the connection form defined on $\mathrm{Sp}(2)$, see \cref{firststructure}.
Since $T_{\theta}$ lies in $\SU(3)$ the torsion transforms trivially and we have 
$\dd \zeta=-\phi\wedge \zeta-[\zeta]\wedge\zeta$
by \cref{eqn: torsion-relation}. So in order to compute the differentials of $\zeta$ we compute the transformed connection one-form $\phi$ from \cref{gauge trafo}. We split $\phi$ into real and imaginary part $\phi=\alpha+i\beta$ to get
\begin{align}
\label{connection-form2}
\alpha=
\begin{pmatrix} 0 & \frac{1}{\sqrt{2}}\re (i\exp(i\theta)\bar{\tau})  & \frac{-1}{\sqrt{2}}\re (\exp(-i\theta)\bar{\tau}) \\ 
\frac{1}{\sqrt{2}}\re (i\exp(i\theta)\tau) & 0 & \frac{1}{2}(3\rho_1+ \rho_2)\cos(2 \theta) \\
\frac{1}{\sqrt{2}}\re (\exp(-i\theta)\tau)& -\frac{1}{2}(3 \rho_1+\rho_2)\cos(2 \theta) & 0        
      \end{pmatrix}
\end{align}
and 
\begin{align}
\label{connection-form1}
\beta=
\begin{pmatrix} -\rho_1+\rho_2 & \frac{1}{\sqrt{2}}\im (i\exp(i\theta)\bar{\tau})  & \frac{-1}{\sqrt{2}}\im (\exp(-i\theta)\bar{\tau}) \\ 
\frac{1}{\sqrt{2}}\im (i\exp(i\theta)\tau) & \frac{1}{2}(\rho_1-\rho_2-2 \dd \theta)& \frac{1}{2} (3\rho_1+ \rho_2)\sin(2\theta) \\
\frac{1}{\sqrt{2}}\im (\exp(-i\theta)\tau) &\frac{1}{2} (3 \rho_1+\rho_2)\sin(2\theta) & \frac{1}{2}(\rho_1-\rho_2 +2 \dd \theta)        
      \end{pmatrix}.
\end{align}
To obtain expressions for the differentials of $\rho_i$ and $\tau$ we use the curvature equation \ref{lagstructureeqns3}.
The curvature tensor $(R+iS)\sigma \wedge \sigma$ transforms tensorially under gauge transformations yielding the explicit expressions
\begin{align}
\label{curvature-theta1}
 S\sigma\wedge \sigma&= \left(
\begin{array}{ccc}
 -\cos (2 \theta ) \sigma_2\wedge \sigma_3 &
   -\frac{1}{2} \cos (2 \theta )
   \sigma_1\wedge \sigma_3 & \frac{1}{2} \cos
   (2 \theta ) \sigma_1\wedge \sigma_2 \\
 -\frac{1}{2} \cos (2 \theta )
   \sigma_1\wedge \sigma_3 & \frac{1}{2} \cos
   (2 \theta ) \sigma_2\wedge \sigma_3 &
   \frac{5}{4} \sin (4 \theta )
   \sigma_2\wedge \sigma_3 \\
 \frac{1}{2} \cos (2 \theta )
   \sigma_1\wedge \sigma_2 & \frac{5}{4} \sin
   (4 \theta ) \sigma_2\wedge \sigma_3 &
   \frac{1}{2} \cos (2 \theta )
   \sigma_2\wedge \sigma_3 \\
\end{array}
\right)
\\
R\sigma\wedge \sigma &= \frac{1}{2}\left(
\begin{array}{ccc}
 0 &  \sigma_1\wedge (\sigma_2- \sin (2\theta )\sigma_3) & 
   \sigma_1\wedge (\sigma_3-
   \sin (2\theta ) \sigma_2) \\
   \sigma_1\wedge (\sin (2 \theta )\sigma_3-\sigma_2) &
   0 & 5\cos ^2(2 \theta )
   \sigma_2\wedge \sigma_3 \\
   \sigma_1\wedge (\sin (2 \theta )\sigma_2- \sigma_3) &
   -5 \cos ^2(2 \theta )
   \sigma_2\wedge \sigma_3 & 0 \\
\end{array}
\right).
\end{align}
Finally, combining the explicit expressions of $\phi, R, S$ with \cref{lagstructureeqns1}-\ref{lagstructureeqns3} results in the following differential identities
\begin{align}
\label{structureeqn-forms}
\begin{split}
 \dd \rho_1&=\frac{3}{2}\cos(2\theta) \sigma_2\wedge\sigma_3, \quad
 \dd \rho_2=\frac{1}{2}\cos(2 \theta) \sigma_2\wedge \sigma_3+i \tau \wedge \bar{\tau} \\
 \dd \tau &=-2i\tau \wedge \rho_2+\frac{1}{\sqrt{2}}\sigma_1\wedge(i \sigma_2 \exp(-i\theta)-\sigma_3 \exp(i\theta))\\
 \dd \sigma_1&=(\epsilon_1\wedge  \sigma_2+\epsilon_2 \wedge \sigma_3)+\sigma_2\wedge \sigma_3 \\
 \dd\sigma_2&=-\frac{1}{2}\cos(2\theta)(3\rho_1+\rho_2)\wedge \sigma_3-\epsilon_1 \wedge \sigma_1-\sigma_1\wedge \sigma_3 \\
 \dd\sigma_3&=\frac{1}{2}\cos(2\theta)(3\rho_1+\rho_2)\wedge \sigma_2- \epsilon_2\wedge \sigma_1+\sigma_1\wedge \sigma_2
\end{split}
\end{align}
with $\epsilon_1=\frac{i}{2\sqrt{2}}(\exp(i \theta)\tau-\exp(-i\theta)\bar{\tau})$ and $\epsilon_2=\frac{1}{2\sqrt{2}}(\exp(-i \theta)\tau+\exp(i\theta)\bar{\tau})$. Recall, that the equation $\beta\wedge \sigma$ implies more algebraic identities.\par 
These differential identities are satisfied on any special Lagrangian in $\CP^3$. Conversely, a special Lagrangian submanifold can locally be reconstructed from such a solution. There is little hope of working out all solutions of \cref{structureeqn-forms}. Instead, one typically imposes additional conditions and then tries to classify all special Lagrangians satisfying the extra condition. For example, one can already see that for $\theta=0$ the equations simplify considerably.
\par
If $\theta\neq \pi/4$ everywhere there is a splitting $TL=E\oplus E^\perp$ where $E^\perp$ is the kernel of the projection $TL\to \mathcal{V}$.
Recall that the standard complex structure $J_1$ on $\CP^3$ agrees with the nearly Kähler structure $J_2$ on $\mathcal{H}$ and differs by a sign on $\mathcal{V}$. 
\begin{prop}
\label{prop: CR submanifolds}
 The distribution $E$ is invariant under the standard complex structure $J_1$ on $\CP^3$ if and only if $\theta\equiv 0$. In that case, $L$ is a CR submanifold for the Kähler structure on $\CP^3$ with $E$ being the $J_1$-invariant distribution on $L$.
\end{prop}
\begin{proof}
 In each point $x\in L$ we can pick a frame $p:T_x M \to \C^3$ such that $TL$ is identified with $W_{\theta(x)}$, $\mathcal{V}$ with $\mathrm{span}(b_3)$ and $J_1$ with $J'$ from \cref{def: J'}. The statement follows from \ref{lma: CR subspaces}. If $\theta=0$ then $E$ is invariant under $J_1$ and $J_1(E^\perp)$ is orthogonal to $TL$, as required.
\end{proof}
CR immersion from $S^3$ to the Kähler $\mathbb{CP}^n$ have been studied in \cite{hu2018equivariant}.
The splitting $TL=E\oplus E^\perp$ gives an ansatz for Lagrangians arising as a product $X^2\times S^1$ such that $TX=E$. Indeed, we will give such an example for $\theta \equiv 0$ later. However, we first show that this ansatz fails when $\theta\neq 0$ and $X$ is compact.
Note that 
\[\omega_{\mathcal{V}}=\frac{i}{2}(\omega_3\wedge \bar{\omega_3})=\frac{1}{2}\cos(2\theta) \sigma_2 \wedge \sigma_3\]
and that $\dd \omega_{\mathcal{V}}$ is a multiple of $\reP$, which vanishes on $L$. This implies the following.
\begin{lma}
 If $\theta\neq \pi/4$ is constant then $\frac{2}{\cos(2\theta)}\omega_{\mathcal{V}}$ defines a calibration on $L$. The fibres of $E$ are the calibrated subspaces of 
 $\frac{2}{\cos(2\theta)}\omega_{\mathcal{V}}$. 
\end{lma}
Since $\omega_{\mathcal{V}}$ is closed on $L$ it defines a cohomology class in $H^2(L,\R)$. We have that when pulled back to $\mathrm{Sp}(2)$, this class vanishes since
$\dd \rho_1=2 \omega_{\mathcal{V}}-\omega_{\mathcal{H}}=3\omega_{\mathcal{V}}$.
If $\theta$ takes values in $(0,\frac{\pi}{4})$ the structure group of $Q$ is just $\mathbb{Z}_2$, generated by $\mathrm{diag}(1,-1,-1)$ in $H$ which corresponds to $\mathrm{diag}(i,i)\in S^1\times S^3 \subset \mathrm{Sp}(2)$. This element leaves $\rho_1$ and $\rho_2$ invariant so in particular $\rho_1$ reduces to a form on $L$ and we have shown.
\begin{prop}
 If $\theta$ takes values in $(0,\pi/4)$ then $[\omega_{\mathcal{V}}]=0$. In particular, in this case $L$ does not have any compact two-dimensional submanifold which is tangent to $E$.
\end{prop}
\subsection{Lagrangians with $\theta \equiv 0$}
The structure equations simplify significantly under the assumption $\theta\equiv 0$. Indeed, plugging $\theta=0$ into \cref{structureeqn-forms} yields $\beta_{32}=0$ and $\dd \theta=0$ implies $\beta_{22}=\beta_{33}$. Since we have $\beta\wedge \sigma=0$ Cartan's lemma implies that there is a single function $f\colon L \to \R$ such that \[\beta_{22}=\beta_{33}=-1/2 \beta_{11}=1/2 f \sigma_1,\quad \beta_{21}=1/2 f \sigma_2,\quad \beta_{31}=1/2 \sigma_3.\] The fundamental cubic is equal to $f(-x_1^3+3/2 x_2^2 x_1+3/2 x_3^2 x_1)$ which has stabiliser $\SO(2)$. Let $\gamma=\rho_1+\rho_2$, then
\[\alpha_{21}=-\frac{1}{2} f \sigma_3, \quad \alpha_{31}=\frac{1}{2} f \sigma_2, \quad \alpha_{32}=-\gamma-\frac{1}{2}f \sigma_1.\]
The structure equations are then equivalent to
\begin{align}
\label{eq: structure equations eqns theta0}
-1 + f + 2 f^2=0, \quad  d \gamma_1= \frac{1}{2} (5 - f) \sigma_2 \wedge \sigma_3.
\end{align}
Hence, there are two examples of special Lagrangian submanifolds $L_f$ with $\theta=0$, both of which are homogeneous and in particular compact. Neither of them is totally geodesic. Note that, as a subset of $\mathrm{Sp}(2)$, the adapted frame bundle is defined by the equations
\begin{align}
\label{theta0eqns}
f \sigma_1=\rho_1-\rho_2,\quad \tau= f\frac{\sigma_2+i\sigma_3}{\sqrt{2}}.
\end{align}
They are both orbits of a Lie group with Lie algebra equal to the span of 
\begin{align*}
\mathfrak{m}_1=\begin{pmatrix} i & 0 \\ 0& i \end{pmatrix}, \quad \mathfrak{m}_2=\begin{pmatrix} i f/{\sqrt{2}} & -1 \\ 1 & -i f/{\sqrt{2}} \end{pmatrix} \quad
\mathfrak{m}_3=\begin{pmatrix} -j &  -j \frac{i}{\sqrt{2}} \\  -j \frac{i}{\sqrt{2}} & j f \end{pmatrix}, \quad \mathfrak{m}_4=\begin{pmatrix} -j i &  j \frac{1}{\sqrt{2}} \\  j \frac{1}{\sqrt{2}} & j i f \end{pmatrix}.
\end{align*}
 Eq. \ref{eq: structure equations eqns theta0} shows that $f$ is in fact constant and must be equal to either $-1$ or $\frac{1}{2}$. So, there are two distinct examples of special Lagrangians with $\theta=0$. We describe the geometry of each of them.
\begin{example}
\label{s1s2}
There is a unique special Lagrangian with $\theta=0$ and $f=-1$. The reduced frame bundle over this submanifold is described by
\[\dd \sigma_1=0,\quad \dd \sigma_2=- \gamma \wedge \sigma_3, \quad \dd \sigma_3= \gamma \wedge \sigma_2, \quad \dd \gamma=3 \sigma_2 \wedge \sigma_3.\]
These are the structure equations of $S^1\times S^2$ where $S^2$ carries a metric of constant curvature $3$.
\end{example} 
\begin{example}
\label{berger}
There is a unique special Lagrangian with $\theta=0$ and $f=\frac{1}{2}$. The reduced frame bundle over this submanifold is described by
\[\dd \sigma_1=\frac{3}{2} \sigma_2\wedge \sigma_3,\quad \dd \sigma_2=- (\gamma+\frac{3}{2} \sigma_1) \wedge \sigma_3, \quad \dd \sigma_3= (\gamma+\frac{3}{2}\sigma_1) \wedge \sigma_2, \quad \dd \gamma=\frac{9}{4} \sigma_2 \wedge \sigma_3.\]
which are the structure equations of a Berger sphere.
\end{example}
\subsection{Lagrangians avoiding Boundary Values}
From now on we assume that $\theta$ is not identically zero or $\pi/4$ and denote by $L^*$ the open set where $\theta$ avoids these values. On $L^*$ the frame bundle reduces to a discrete bundle. For $J$-holomorphic curves in $\cp$ the second fundamental form is determined by two angle functions \cite{aslan2021transverse}. In contrast, $\theta$ alone does not determine the second fundamental form of $L^*$.
We let $\dd \theta=t_1 \sigma_1+t_2 \sigma_2+t_3 \sigma_3$ and $x=h_{221},y=h_{222},z=h_{322},w=h_{321}$ such that $\beta$ is entirely determined by $\theta$ and these quantities.
Clearly all of these functions are constant on orbits of Lie subgroups of $\mathrm{Sp}(2)$, the converse is also true.
\begin{prop}
\label{prop: homo-solutions}
 Any solution with $x,y,z,w$ and $\theta\in (0,\pi/4)$ constant is an orbit of a Lie group.
\end{prop}
\begin{proof}
 Let $L$ be the special Lagrangian corresponding to this solution with adapted frame bundle $\hat{L}$. Then $\hat{L}$ is an integral submanifold of the EDS generated by $\eta_i,\beta-h \sigma$. By assumption, $h$ has constant coefficients which means that the equations $\eta_i=0,\beta=h \sigma$ describe a linear subspace of $\mathfrak{sp}(2)$ and hence $\hat{L}$ is a Lie group and $\hat{L}\to L$ is a double cover of Lie groups.
\end{proof}
 In principle, we could derive a set of PDE's for $\theta,x,y,z$ and $t_i$ from the structure equations but this is not practical in full generality. 
 However, there is a somewhat surprising homogeneous example.
 \begin{example}
 \label{exotic-solution}
 Setting 
\[x=-\sqrt{2/5},\quad y=0,\quad z=0,\quad w=-\frac{3}{5}\sqrt{3/2},\quad \theta=\frac{1}{2}\arccos(\frac{7 \sqrt{2}}{5 \sqrt{5}})\] is a solution to \cref{structureeqn-forms} and hence corresponds to a unique special Lagrangian in $\CP^3$. The fundamental cubic of this example is given by $\sqrt{\frac{2}{5}}(2x_1^3-3x_1 x_2^2-3x_1x_3^2)-\frac{9}{5}\sqrt{6}x_1x_2x_3,$
whose orientation preserving symmetry group is $\mathbb{Z}_2$ coming from $(x_2,x_3)\mapsto (-x_2,-x_3)$. The Ricci curvature is diagonal in the (dual) frame $\sigma_1,\sigma_2+\sigma_3,\sigma_2-\sigma_3$ in which $\mathrm{Ric}=\mathrm{diag}(-99/50,-27/50(-2+\sqrt{15}),27/50(2+\sqrt{15})).$
\end{example}
By \cref{prop: homo-solutions}, this example is homogeneous.
We have found new examples by imposing conditions on $\theta$. In each case, the fundamental cubic has non-trivial symmetries. The structure equations \cref{structureeqn-forms} only hold in a fixed gauge. This makes it difficult to classify special Lagrangians where the fundamental cubic has a symmetry everywhere. We do not have the gauge freedom to bring them into the standard form as in \cite[Proposition 1]{bryant2006second}. However, this poses no problem for the totally geodesic case.
\begin{prop}
\label{prop: totally geodesic class}
 Up to isometries, the standard $\mathbb{RP}^3$ is the unique totally geodesic special Lagrangian in $\CP^3$. 
\end{prop}
\begin{proof}
There is no totally geodesic Lagrangian which lies in $\theta\in [0,\frac{\pi}{4})$. This is because in that case $\beta=0$ forces $\rho_1=0$ but this is a contradiction to the first equation of \ref{structureeqn-forms}. \par
If $L$ is a totally geodesic Lagrangian with $\theta\equiv \frac{\pi}{4}$ then the adapted frame bundle $Q$ is a four-dimensional submanifold of $\mathrm{Sp}(2)$ on which $\eta_i$ and $\beta$ vanish. If $\theta\equiv \frac{\pi}{4}$ then $S$ vanishes on the adapted bundle $Q$ and by \cref{lagstructureeqns3} the ideal generated by $\eta_i$ and $\beta_{ij}$ is closed under differentials. By Frobenius' theorem, there is a unique maximal submanifold on which these forms vanish that passes through the identity $e\in \mathrm{Sp}(2)$. Hence, up to isometries, there is a unique totally geodesic special Lagrangian in $\CP^3$ with $\theta=\frac{\pi}{4}$. We have already found this example, it is the standard $\mathbb{RP}^3\subset \CP^3$.
\end{proof}
\section{Classifying $\SU(2)$ invariant Special Lagrangians}
\label{sec: SU2 SL}
Instead of imposing symmetries on the fundamental cubic, we shall now impose them on the special Lagrangian itself. We have already encountered examples of homogoneous special Lagrangians.\par There are examples of special Lagrangians admitting a cohomogeneity one action of $\SU(2)$ in both $S^6$ and $\C^3$. In $S^6$, there is a unique example of this type, the squashed three-sphere \cite[Example 6.4]{lotay2011ruled}. In $\C^3$, the Harvey-Lawson examples \cite{bryant2006second, harvey1982calibrated}
\[L_c=\{(s+it)u\mid u\in S^2\subset \R^3, t^3-3s^2t=c^3\}\]
admit a cohomogeneity one action of $\SO(3)$ for $c\neq 0$. \par 
The situation in $\CP^3$ is different.
We show in this section that all special Lagrangians that admit an action of an $\SU(2)$ group of automorphism are in fact homogeneous and have already been described in the previous section.
We introduce $\SU(2)$ moment-type maps to prove this classification.
\subsection{$\SU(2)$ Moment Maps}
\label{subsec: SU2 moment-type maps}
Assume that $\SU(2)$ acts effectively on $M$ with three-dimensional principal orbits and by nearly Kähler automorphisms.
Let $\{\xi_1,\xi_2,\xi_3\}$ be a basis of $\mathfrak{su}(2)$ such that $[\xi_i,\xi_j]=-\epsilon_{ijk} \xi_k$. 
Denote the corresponding fundamental vector fields by $K^{\xi_i}$. The map $\xi\to K^{\xi}$ is an anti Lie algebra homomorphism. Hence, the vector fields $K^{\xi_i}$ obey the standard Pauli commutator relationships $[K^{\xi_i},K^{\xi_j}]=\epsilon_{ijk} K^{\xi_k}$.
Consider the map \[\mu=(\mu_1,\mu_2,\mu_3)=(\omega(K^{\xi_2},K^{\xi_3}),\omega(K^{\xi_3},K^{\xi_1}),\omega(K^{\xi_1},K^{\xi_2})).\] Then $\mu\colon M\to \R^3$ is an $\SU(2)$ equivariant map with respect to the action of $\SU(2)$ on $\R^3$ coming from the double cover $\SU(2)\to \SO(3)$. In addition, define the invariant scalar function \[\nu=\imP(K^{\xi_1},K^{\xi_2},K^{\xi_3}).\]

 The map $\mu$ is not a multi-moment-type map in the sense of \cite[Definition 3.5]{madsen2013closed}. The Lie-kernel of $\Lambda^2 \mathfrak{su}(2)\to\mathfrak{su}(2)$ is trivial, so there is no non-trivial multi-moment map for the three form $\reP$. On the other hand, the map $\Lambda^3 \mathfrak{su}(2)\to \Lambda^2 \mathfrak{su}(2)$ is trivial and $\nu$ is a multi-moment map with values in $\R\cong \Lambda^3 \mathfrak{su}(2)$ for $-\frac{1}{2}\omega\wedge \omega$. We will refer to $\mu$ and $\nu$ as multi-moment-type maps.\par
The general strategy to obtain moment-type maps is to contract Killing vector fields with the nearly Kähler forms. Using a standard argument, the following lemma shows that all such combinations are exhausted by $\mu$ and $\nu$.
\begin{lma}
 The form $\reP$ vanishes on $SU(2)$ orbits, i.e. $\reP(K^{\xi_1},K^{\xi_2},K^{\xi_3})=0$.
\end{lma}
\begin{proof}
 Let $\mathcal{O}$ be a three-dimensional orbit of $SU(2)$. Since $SU(2)$ acts by isometries on $M$ we have that $\mathrm{vol}_\mathcal{O}$ is a $SU(2)$ invariant form on $\mathcal{O}$. The same holds for $\reP|_\mathcal{O}$. So there is $\lambda \in \R$ such that $ \reP|_\mathcal{O} =\lambda \mathrm{vol}_{\mathcal{O}}$. Since $\reP$ is exact 
 $\lambda \mathrm{vol}(\mathcal{O})=\int_{\mathcal{O}} \reP=0$
 i.e. $\lambda=0$.
\end{proof}
 Since $\psi=\reP+ i \imP$ is non-degenerate this means that $\nu$ vanishes if and only if $K^{\xi_1},K^{\xi_2},K^{\xi_3}$ are linearly dependent over $\mathbb{C}$. By Cartan's formula and the nearly Kähler structure equations we get
 \begin{align}
 \label{d nu}
  \dd \nu=2 \sum_l \mu_l \omega(K^{\xi_l},\cdot), \text{ and } \dd \mu_k&=-\omega(K^{\xi_k},\cdot)+3 \reP(K^{\xi_i},K^{\xi_j},\cdot)
 \end{align}
where $(i,j,k)$ is a cylic permutation of $(1,2,3)$. The following proposition is somewhat similar to the toric situation \cite{dixon2019multi} as we can identify $\SU(2)$ invariant special Lagrangians orbits by the values of the maps $\mu$ and $\nu$.
\begin{prop}
The orbit of a point $x\in M$ is special Lagrangian if and only if $\nu(x)\neq 0$ and $\mu(x)=0$. The set $\nu^{-1}(0)\cap \mu^{-1}(0)$ is a union of fixed points of the $SU(2)$ action and two-spheres on which $\omega$ vanishes.
If $M$ has non-vanishing Euler characteristic then $0$ lies in the image of $\nu$.
The function $\nu$ is not constant and the set of points in which $\dd \nu=0$ and $\nu \neq 0$ consists of special Lagrangian orbits.
\end{prop}
\begin{proof}
By the definition of $\mu$, the two-form $\omega$ vanishes on the $\SU(2)$ orbit of $x$ if and only if $\mu(x)=0$. If $\nu(x)\neq 0$ then $K^{\xi_i}$ are linearly independent at $x$ and the orbit at $x$ is 3-dimensional, which implies the first statement.
If $\nu(x)=0$ then the orbit has dimension less than three and the second statement follows from the fact that lower-dimensional $\SU(2)$ orbits must be points or two-spheres. \par
Eq. \ref{d nu} implies that if $\nu$ is constant then $(K^{\xi_1},K^{\xi_2},K^{\xi_3})$ are linearly dependent everywhere which contradicts the principal orbit type being three-dimensional. If $\chi(M)\neq 0$ then any vector field $K^{\xi_i}$ must have a zero, which forces $\nu$ to vanish.
Finally, consider a point $x$ in which $\dd \nu=0$ and $\nu \neq 0$. We want to show that $\mu(x)=(0,0,0)$. Using the action of $SU(2)$ we can assume that $\mu_2(x),\mu_3(x)=0$. Then $0=JK^{\xi_1}\nu=-2 \|K^{\xi_1}\|^2\mu_1$. But $\nu\neq 0$ and hence $\mu_1(x)=0$.
\end{proof}
Since either the maximum or minimum of $\nu$ is not zero this implies an existence result for special Lagrangians.
\begin{crl}
\label{crl: existence SL orbits}
 If $M$ is compact then the $SU(2)$ action has a special Lagrangian orbit.
\end{crl}
If $L$ is a special Lagrangian submanifold on which a $\SU(2)$ subgroup acts then $L$ will lie in the vanishing set of $\mu$. 
So we can classify all $\SU(2)$ invariant special Lagrangian submanifolds of $\CP^3$ by computing the vanishing set of $\mu$ for every $\SU(2)$ subgroup of $\mathrm{Sp}(2)$.
 \begin{dfn}
 \label{dfn: subgroups}
  Define the three $\SU(2)$ subgroups of $\Sp(2)$ as $K_1=\{1\}\times\mathrm{Sp}(1)$, $K_2=\SU(2)$, arising from the inclusion $\C^2\subset \HH^2$, and $K_3$ which comes from the irreducible representation of $\SU(2)$ on $S^3(\C^2)\cong \C^4$.
 \end{dfn}
Any three-dimensional subgroup of $\mathrm{Sp}(2)$ is conjugate to one of $K_1,K_2,K_3$.
\begin{rmk}
Note that $\SO(4)$ contains two $\SU(2)$ subgroups that do not stabilise a vector in $\R^4$. They are not conjugated to each other and, on the Lie algebra level, correspond to the splitting of $\Lambda^2(\R^4)$ into self-dual and anti-self-dual two forms. However, in $\SO(5)$, these two Lie algebras are conjugated to each other, for example via the element $(x_4,x_5)\mapsto (-x_4,-x_5)$. Since $\mathrm{Sp}(2)=\mathrm{Spin}(5)$ the same holds true for the corresponding $\SU(2)$ subgroups in $\mathrm{Sp}(2)$.
\end{rmk}
The groups $K_i$ naturally act on $S^4$ through the double cover $\Sp(2)\to \SO(5)$.
 The group $K_1$ acts via $\SU(2)\subset \SO(4)$, the group $K_2$ via the double cover $\SU(2)\to \SO(3)$ leaving a plane in $\R^5$ invariant and $K_3$ acts irreducibly on $\R^5$ and factors through $\SO(3)$.
To relate the group invariant examples to those found in the previous section we compute the function $\theta$ for group orbits, for which we use \cref{eq: formula for theta}. To this end, it makes sense to define $\mu_{\mathcal{V}}=(\omega_{\mathcal{V}}(K^{\xi_2},K^{\xi_3}),\omega_{\mathcal{V}}(K^{\xi_3},K^{\xi_1}),\omega_{\mathcal{V}}(K^{\xi_1},K^{\xi_2})).$
The Killing vector fields corresponding to the subgroups $K_i$ admit quite simple expressions in local coordinates. So, to express $\nu$ for $K_i$ in homogeneous coordinates we need to do so for the nearly Kähler form $\omega=\frac{i}{2}\sum_i \omega_i \wedge \bar{\omega}_i$. This is the essence of \cref{eq: pull back local frame}, where the forms $\omega_i$ are pulled back to a chart in $\CP^3$ by a local section.\par
It will be challenging to compute $\nu$ for $K_2$ and $K_3$, so we first establish  representation theoretic results to simplify the computations.
In \cite{galindo2002two}, it is shown that given an irreducible finite-dimensional continuous real representation of a compact Lie group $G$, the intersection of any hyperplane and any group orbit is non-empty. The authors of \cite{galindo2002two} pose the question whether the same statement holds for complex representations, in particular irreducible representations of $\SU(2)$. There is a general framework to relate this question to the existence of nowhere vanishing sections in bundles over the flag manifold $G/T$ \cite{an2010universal}. The following result follows a similar strategy and gives a direct proof for $G=\SU(2)$.
\begin{lma}
\label{repsu2}
 Let $(V,\rho)$ be a finite dimensional unitary representation of $G=\SU(2)$ with all weights non-zero and $H$ be a hyperplane which is invariant under the maximal torus $\mathrm{U}(1)\subset \SU(2)$. Then $H$ intersects every $G$ orbit. 
\end{lma}
\begin{proof}
 Since $H$ is $\mathrm{U}(1)$ invariant there is a linear $\mathrm{U}(1)$ equivariant map $f\colon V \to \C$ such that $\mathrm{ker}(f)=H$. Assume that there is an $x \in V$ such that $G.x\cap H=\emptyset$. Then $s\colon g\mapsto f(gx)$ is a non-vanishing $\mathrm{U}(1)$ equivariant map $\SU(2)\to \C$. Restricting this map to $\mathrm{U}(1)\subset\SU(2)$ gives a representation $\tau$ of $\mathrm{U}(1)$ on $\C$ of weight $k\in \mathbb{Z}$.\par Note that the principal bundle $\SU(2)\to \SU(2)/{\U(1)}=S^2$ is the Hopf fibration and that $s$ gives rise to a nowhere vanishing section of the associated bundle $E=\SU(2)\times_{\tau} \C$ over $S^2$. Since the Hopf fibration has non-trivial Chern class, the complex line bundle $E$ is trivial which forces $k=0$. This is a contradiction because $f$ restricts to an equivariant isomorphism from $H^\perp$ to $\C$, so $H^\perp$ is a zero-weight subspace.
\end{proof}
Note that, in the situation above, $H$ is invariant under $\mathrm{U}(1)$ and the action of $\mathrm{U}(1)$ on $H$ splits into one-dimensional components. Then every $G$ orbit also intersects the set $H'\subset H$ where one of the $\C$ components is restricted to the set $\R_{\geq 0}$.  \par All the actions of $K_i$ on $\CP^3$ factor through an action of $\SU(4)$ on $\C^4$. The irreducible action $\rho_k$ of $\SU(2)$ on $S^k(\C^2)$ has weights $(k,k-2,\dots,-k+2,-k)$. The action of $K_2$ on $\CP^3$ factors through $\rho_1\oplus \rho_1$ on $\C^4$ and $K_3$ through $\rho_3$ on $\C^4$. In particular neither has a zero weight, so \cref{repsu2} applies to these cases.
\subsubsection{$K_1$}
Recall $K_1=\{1\}\times \mathrm{Sp}(1)$, we compute the Killing vector fields on the chart $\mathbb{A}_0=\{Z_0 \neq 0\}$
\[K^{\xi_1}=-\im(Z_2 \frac{\partial}{\partial Z_2}-Z_3 \frac{\partial}{\partial Z_3}), \quad K^{\xi_2}=\re(Z_3 \frac{\partial}{\partial Z_2}-Z_2 \frac{\partial}{\partial Z_3}), \quad K^{\xi_3}=\im(Z_3 \frac{\partial}{\partial Z_2}+Z_2 \frac{\partial}{\partial Z_3}).\]
We contract these vector fields with the nearly Kähler forms $\omega$ and $\psi$ in homogeneous coordinates from \cref{eq: pull back local frame},
which gives
\begin{align*}
 \nu&=-|Z|^{-6}\frac{1}{2}(|Z_0|^2+|Z_1|^2)(|Z_2|^2+|Z_3|^2)^2,\quad \mu_1=|Z|^{-2}(|Z_3|^2-|Z_2|^2) f\\
 \mu_2+i \mu_3&= 2 i |Z|^{-2} Z_2 \overline{Z_3} f,\quad f=\frac{1}{4}|Z|^{-2}(-2(|Z_0|^2+|Z_1|^2)+(|Z_2|^2+|Z_3|^2)).
\end{align*}
Hence, $\nu$ vanishes on the line of fixed points $\{Z_2=Z_3=0\}$ or when $f=0$. Note that $\mathrm{Sp}(1)\times\mathrm{Sp}(1)$ is the centraliser of $K_1$ in $\mathrm{Sp}(2)$, acts with cohomogeneity one on $\CP^3$ and the orbits of that action are the level sets of $f$. In particular $\mathrm{Sp}(1)\times \mathrm{Sp}(1)$ acts transitively on $f=0$ which means that up to isometries there is a unique special Lagrangian on which $K_1$ acts. Hence, for simplification we consider the orbit $\mathcal{O}_{11}$ at the point $P_{11}=[1,0,\sqrt{2},0]$. At this point, $K^{\xi_1}$ annihilates $\omega_{\mathcal{V}}$ which means that evaluating \cref{eq: formula for theta} at $P_{11}$ yields 
\[\frac{1}{2} |\cos(2\theta)|=\|K^{\xi_1}\| |\frac{\omega_{\mathcal{V}}(K^{\xi_2},K^{\xi_3})}{\nu}|=\frac{1}{2}.\]
Hence $\theta=0$ and 
$\mathcal{O}_{11}$ is diffeomorphic to $S^3$. It is also the orbit of the larger group $S^1\times \mathrm{Sp}(1)$.
\begin{lma}
 The unique special Lagrangian invariant under $K_1$ is $\mathcal{O}_{11}$ which is identified with \cref{berger}.
\end{lma}
\begin{rmk}
The multi-moment map for $\T^2$ torus symmetry is an eigenfunction of the Laplace operator on $M$, cf. \cite[Lemma 3.1]{russo2019nearly}. However, this is not the case for the $\SU(2)$ multi-moment map $\nu$, as it is non-positive everywhere, so
$\int_M \nu \mathrm{vol}_M <0$, this integral vanishes for eigenfunctions of Laplace operator.
\end{rmk}
\subsubsection{$K_2$}
The group $K_2$ lies inside $\mathrm{U}(2)\subset \mathrm{Sp}(2)$. Let $\xi_0=\mathrm{diag}(i,i)\in \mathfrak{sp}(2)$, which commutes with all elements in the Lie algebra of $K_2$. Again, we compute
\begin{align*}
K^{\xi_0}&=2\re(i Z_1 \frac{\partial}{\partial Z_1}) \qquad K^{\xi_1}=2\re(-2i(Z_1 \frac{\partial}{\partial Z_1}+Z_2 \frac{\partial}{\partial Z_2})), \\
K^{\xi_2}+iK^{\xi_3}&=2(Z_3-Z_1 Z_2 \frac{\partial}{\partial Z_1}-(1+Z_2^2) \frac{\partial}{\partial Z_2}-(Z_1+Z_2 Z_3) \frac{\partial}{\partial Z_3}).
\end{align*}
For $K_2$, the map $\nu$ is equal to
\[-8\frac{|Z_0 Z_1+Z_2 Z_3|^2}{|Z|^2}\]
by \cref{eq: pull back local frame}.
We apply \cref{repsu2} and compute $\mu$ on the set $Z_2=0$ and $Z_1=r\geq 0$, and w.l.o.g we assume $Z_0=1$.
Then we have 
\begin{align*}
 \mu_1 =-2|Z|^{-4} (-1+r^4+4 |Z_3|^2-|Z_3|^4), \quad \mu_2- i \mu_3=-4 i |Z|^{-4}rZ_3(-2+r^2+|Z_3|^2).
\end{align*}
The set $\nu=0$ is a $J_1$-holomorphic quadric and hence diffeomorphic to $S^2\times S^2$. The action of $\mathrm{U}(2)$ on this quadric is of cohomogeneity one. The principal orbit is $S^2\times S^1$ and the singular orbit $S^2$.\par If $\nu=0$ then $\nu$ vanishes if $r=0$ and $|Z_3|=\sqrt{2+\sqrt{3}}$. Denote this point by $P_{21}$, the $\mathrm{U}(2)$ orbit $\mathcal{O}_{21}$ is special Lagrangian and $K^{\xi_0}$ is horizontal on $\nu^{-1}(0)$. We compute via \cref{eq: formula for theta} 
\[\frac{1}{2} |\cos(2\theta)|=\|K^{\xi_0}\| \frac{|\omega_{\mathcal{V}}(K^{\xi_2},K^{\xi_3})|}{|\imP(K^{\xi_0},K^{\xi_2},K^{\xi_3})|}=\frac{1}{2},\]
i.e. $\theta=0$.\par
If $\nu \neq 0$ then $r\neq 0$ and $\nu=0$ only occurs for $Z_3=0$ and $r=1$. Denote this point by $P_{22}$ and note $\nu_{\mathcal{V}}$ vanishes on $P_{22}$. Hence, $\theta=\pi/4$ and the orbit $\mathcal{O}_{22}$ is diffeomorphic to $\mathbb{RP}^3$.
\begin{lma}
 All special Lagrangians that admit a $K_2$ action are $\mathcal{O}_{21}, \mathcal{O}_{22}$ which corresponds to \cref{s1s2} and \cref{standardRP3} respectively.
\end{lma}
\subsubsection{$K_3$}
To compute the Killing vector fields for $K_3$ we need the explicit description of $K_3\subset \SU(4)$  
\begin{align*} 
K_3=\left\{\left( 
\begin{array}{cccc}
a^{3}               & -\sqrt{3} a^{2} \overline{b} & \sqrt{3} a \overline{b}^{2}        & - \overline{b}^{3}\\
\sqrt{3} a^{2}b  & a (|a|^{2} - 2 |b|^{2})         & - \overline{b} (2|a|^{2} - |b|^{2}) & \sqrt{3} \overline{a} \overline{b}^{2}\\
\sqrt{3} a b^{2} & b (2|a|^{2} - |b|^{2})          & \overline{a} (|a|^{2} - 2 |b|^{2})  &-\sqrt{3} \overline{a}^{2} \overline{b}\\
b^{3}               & \sqrt{3} \overline{a} b^{2}  & \sqrt{3} \overline{a}^{2} b        & \overline{a}^{3}
\end{array} 
\right)
\mid (a,b)\in S^3\subset \C^2 \right\},
\end{align*}
see for example \cite{kawai2018second}. Now we can compute the Killing vector fields for $K_3$ on $\mathbb{A}_0$
\begin{align*}
K^{\xi_1}=&2\im(3Z_1 \frac{\partial}{\partial Z_1}+2Z_2 \frac{\partial}{\partial Z_2}+Z_3 \frac{\partial}{\partial Z_3}) \\
K^{\xi_2}=&\re(-\sqrt{3}(Z_1 Z_3+Z_2) \frac{\partial}{\partial Z_1}+(\sqrt{3}Z_1-(2+\sqrt{3}Z_2)Z_3) \frac{\partial}{\partial Z_2}\\
&+(2Z_2-\sqrt{3}(1+Z_3^2)) \frac{\partial}{\partial Z_3} ) \\
K^{\xi_3}=&\im(\sqrt{3}(-Z_1 Z_3+Z_2) \frac{\partial}{\partial Z_1}+(\sqrt{3}Z_1+(2-\sqrt{3}Z_2)Z_3) \frac{\partial}{\partial Z_2}\\
&+(2Z_2-\sqrt{3}(-1+Z_3^2)) \frac{\partial}{\partial Z_3}).
\end{align*}
Again, we apply \cref{repsu2} and restrict ourselves to compute $\nu$ and $\nu$ for $Z_0=1,Z_2=r>0$ and $Z_3=0$. Let furthermore $Z_1=\exp(i\phi)s$, then by \cref{eq: pull back local frame}
\begin{align*}
 \mu_1&=2|Z|^{-4} \left(5 r^4-4 r^2 s^2-16 r^2-3 s^4+3\right)\\ 
  \mu_2&=|Z|^{-4}4 r s \sin(\phi ) \left(r (\sqrt{3} r-9)+\sqrt{3} \left(s^2-8\right)\right) \\
  \mu_3&= |Z|^{-4} 4 r s \cos(\phi ) \left(r (-\sqrt{3} r-9)-\sqrt{3} \left(s^2-8\right)\right)  \\
 \nu&=|Z|^{-3}8 \left(4 r^4 \left(s^2-5\right)-12 \sqrt{3} r^3 s^2 \cos (2 \phi )+3 r^2
   \left(s^2+4\right)-9 \left(s^4+s^2\right)\right).
\end{align*}
Hence, the only solutions of $\mu=(0,0,0)$ are $(r,s)\in\{(0,1),(\sqrt{3},0),(1/{\sqrt{5}},0)\}$. The solutions with $r=0$ are in the $\mathrm{U}(1)$ orbit of the point $P_{31}=[1,0,1,0]$. The point $[1,\sqrt{3},0,0]$ is also in the same $K_3$ orbit as $P_{31}$. So, it suffices to consider the points $P_{31}$ and $P_{32}=[1,1/{\sqrt{5}},0,0]$.\par
Note that $\nu(P_{31})=-18$ and $\nu(P_{32})=200/27$ which must hence be the minimum and maximum of $\nu$ respectively. The map $\mu_{\mathcal{V}}$ vanishes at $P_{31}$ and hence the orbit $\mathcal{O}_{31}$ satisfies $\theta=0$ and is in fact the Chiang Lagrangian. 
\par
Furthermore, $\mu_{\mathcal{V}}(P_{32})=(-\frac{14}{9},0,0)$ which means that $K^{\xi_1}$ is horizontal at $P_{32}$. By \cref{eq: formula for theta} we have
\[\frac{1}{2} |\cos(2\theta)|=\|K^{\xi_1}\| \frac{|\nu_{\mathcal{V}}|}{|\nu|}=\frac{7}{5\sqrt{10}}, \quad \theta=\frac{1}{2}\arccos(\frac{7 \sqrt{2}}{5 \sqrt{5}})\approx 0.24\]
on $\mathcal{O}_{32}$.
\begin{lma}
 All $K_3$ invariant special Lagrangians are given by the orbits $\mathcal{O}_{31}$ and $\mathcal{O}_{32}$, which correspond to \cref{Chiang-Lagrangian} and \cref{exotic-solution} respectively.
\end{lma}
As remarked after \cref{prop: nk class}, the identity component of nearly Kähler automorphisms of $\cp$ is $\Sp(2)$.
So, combining all results of this section results in the following theorem.
\begin{thm}
\label{thm: classification SL orbits}
Every Special Lagrangian in $\CP^3$ that admits a non-trivial action of a three-dimensional group of nearly Kähler automorphisms is homogeneous and one of the following orbits.
\\
\begin{tabular}{c c c c c}
 Example & Properties & $\theta$ & Group orbit & Stabiliser group of $C$ \\ 
 \hline
 \ref{berger}&Berger Sphere & 0 & $K_1$ & $\SO(2)$ \\
 \ref{s1s2}&$S^1\times S^2$ & 0 & $\U(2)\supset K_2$ & $\SO(2)$ \\
 \ref{standardRP3}&standard $\mathbb{RP}^3$ & $\pi/4$  & $K_2$ & $\SO(3)$ (tot. geodesic)\\
 \ref{Chiang-Lagrangian}& Chiang Lagrangian & $\pi/4$ & $K_3$ & $S_3$ \\
 \ref{exotic-solution} & distinct $\mathrm{Ric}$ e'values &$\approx 0.24$ & $K_3$ & $\mathbb{Z}_2$
\end{tabular}
\end{thm}
For the definition of the $\SU(2)$ subgroups $K_i$ see \cref{dfn: subgroups}.
\subsection{The Flag Manifold} 
\Cref{thm: classification SL orbits} classifies homogeneous special Lagrangians and also rules out the existence of special Lagrangians admitting a cohomogeneity one action of a three-dimensional group of nearly Kähler automorphisms. The aim of this section is to prove the analogous statement for the nearly Kähler flag manifold $\mathbb{F}=\SU(3)/{\T^2}$. The homogeneous special Lagrangians in the flag manifold are classified in \cite{storm2020lagrangian}, so we restrict ourselves to the cohomogeneity-one case.\par
We could achieve this by computing the moment-type maps and determine the zero sets, as we did for $\cp$. However, the statement can also be shown by analysing the group actions of 3-dimensional subgroups of $\SU(3)$, which is the identity component of nearly Kähler automorphisms as remarked after \cref{prop: nk class}. We will show the set of elements with one-dimensional stabilisers are two-dimensional, so they cannot be special Lagrangian. To understand the action of three-dimensional subgroups of $\SU(3)$ on the flag manifold we exploit the fact that the flag manifold is an adjoint orbit for $\SU(3)$. \par
Up to conjugation there are two three-dimensional subgroups of $\SU(3)$ the standard $\SO(3)\subset \SU(3)$ and the $\SU(2)$ subgroup fixing the element $(0,0,1)$.
Consider the adjoint action of $\SU(3)$ on its lie algebra $\mf{su}(3)$. Every element $A\in \mf{su}(3)$ is then conjugate to a diagonal matrix, the $\SU(3)$ orbits are distinguished by the set of purely imaginary eigenvalues.\par
Hence, the $\SU(3)$ orbits are the level sets of the functions
\begin{align*}
 \rho_1(A)=\mathrm{Tr}(A^2) \quad \rho_2(A)=\mathrm{Tr}(A^3).
\end{align*}
There are three orbit types. The principal stabiliser type is a maximal torus in $\SU(3)$. Every element with distinct eigenvalues is of principal type. If $A$ has a repeated eigenvalue the stabiliser type is $\SU(2)\times \U(1)$, unless all eigenvalues are zero. \par
So we fix an element $A\in\mf{su}(3)$ with distinct eigenvalues and identify the flag manifold $\SU(3)/\mathbb{T}^2$ with the adjoint orbit of $A$.
Our aim is to determine the set of elements in $\mathbb{F}$ with one-dimensional stabiliser under the action of $\SO(3)\subset \SU(3)$ and $\SU(2)$. 
\begin{prop}
 Every element in $\mathfrak{su}(3)$ with non-principal stabiliser for either the action of $\SO(3)$ or $\SU(2)$ has a representative in the set
 \[S=\left\{ \left(
\begin{array}{ccc}
 i \mu  & -\lambda  & 0 \\
 \lambda  & i \mu  & 0 \\
 0 & 0 & -2 i \mu  \\
\end{array}
\right) \mid \lambda,\mu \in \R \right\}. \]
\end{prop}

\begin{proof}
 With respect to $\SO(3)$ the adjoint action on $\mf{su}(3)$ splits as $\Lambda^2(\R^3)\oplus i S^2_0(\R^3)$. It is known that the action of $\SO(3)$ on $S^2_0(\R^3)$ is irreducible and has trivial stabiliser unless the element in $S^2_0(\R^3)$ has repeated eigenvalues, in which case the stabiliser is $\O(2)$. Every such element is conjugate to $\mathrm{diag}(\mu,\mu,-2\mu)$. Let $A\in \Lambda^2(\R^3)$, the stabiliser of $A$ in $\SO(3)$ intersects the stabiliser of $\mathrm{diag}(\mu,\mu,-2\mu)$ in a one-dimensional set if and only if $A$ is of the form 
 \[A_{\lambda}=\left(
\begin{array}{ccc}
 0  & -\lambda  & 0 \\
 \lambda  & 0  & 0 \\
 0 & 0 & 0  \\
\end{array}
\right), \]
which implies the statement for $\SO(3)$.\par With respect to the subgroup $\SU(2)$, the representation splits as $\mf{su}(2)\oplus \C^2 \oplus \R$ where the action on the first summand is the adjoint action, is irreducible on the second summand and trivial on the third summand. For the stabiliser to be non-trivial the component in $\C^2$ has to vanish. The trivial component is spanned by the element $\mathrm{diag}(\mu,\mu,-2\mu)$. Finally, every element in $\mf{su}(2)$ is conjugate to $A_{\lambda}$ under the $\SU(2)$ action.
\end{proof}
\begin{thm}
There are no special Lagrangians in $\mathbb{F}$ which admits a cohomogeneity one action of nearly Kähler automorphisms.
\end{thm}
\begin{proof}
We show that for a three-dimensional group acting by nearly Kähler automorphisms on the flag manifold, the set of elements in $\mathbb{F}$ with one-dimensional stabiliser is two-dimensional.\par 
The identity component of nearly Kähler automorphisms of the flag manifold is $\SU(3)$. Since $\SU(2)$ and $\SO(3)$ are the only three-dimensional subgroups of $\SU(3)$ it suffices to show that the intersection $\mathbb{F} \cap S$ is finite.
 Since $S$ is two-dimensional it only remains to check that the function $\rho=(\rho_1,\rho_2)$ has full rank on $S$. A direct computation shows that the determinant of the Jacobian is
 \[\det (J_\rho) =  24 i \lambda ^3-216 i \lambda  \mu ^2\]
 which vanishes if and only if $\lambda\in \{0,3 \mu,- 3\mu\}$. In each case, the resulting element in $S$ has repeated eigenvalues, so it does not lie in $\mathbb{F}$ since $A$ has distinct eigenvalues.
\end{proof}
\printbibliography
\Addresses
\end{document}